\documentclass[reqno]{amsart}

\usepackage{mathrsfs, color}
\usepackage{amsfonts}
\usepackage{amssymb,amsmath}
\usepackage{amsthm}

\newtheorem{theorem}{Theorem}[section]
\newtheorem{lemma}[theorem]{Lemma}

\theoremstyle{definition}
\newtheorem{definition}[theorem]{Definition}

\newtheorem{proposition}[theorem]{Proposition}

\theoremstyle{remark}
\newtheorem{remark}[theorem]{Remark}

\numberwithin{equation}{section}

\begin{document}

\title[Bayesian inverse problems]
 {Bayesian approach to inverse problems for functions with variable index Besov prior}

\author[J.X.Jia]{Junxiong Jia}
\address{Department of Mathematics,
Xi'an Jiaotong University,
 Xi'an
710049, China;
Beijing Center for Mathematics and Information Interdisciplinary Sciences (BCMIIS)}
\email{jjx323@mail.xjtu.edu.cn}
\thanks{}

\author[J. Peng]{Jigen Peng}
\address{Department of Mathematics,
Xi'an Jiaotong University,
 Xi'an
710049, China;
Beijing Center for Mathematics and Information Interdisciplinary Sciences (BCMIIS)}
\email{jgpeng@mail.xjtu.edu.cn}

\author[J. Gao]{Jinghuai Gao}
\address{School of Electronic and Information Engineering,
Xi'an Jiaotong University,
 Xi'an
710049, China;
Beijing Center for Mathematics and Information Interdisciplinary Sciences (BCMIIS)}
\email{jhgao@mail.xjtu.edu.cn}


\subjclass[2010]{}

\date{}

\keywords{}

\begin{abstract}
We adopt Bayesian approach to consider the inverse problem of estimate a function from noisy observations.
One important component of this approach is the prior measure.
Total variation prior has been proved with no discretization invariant property, so Besov prior has been proposed recently.
Different prior measures usually connect to different regularization terms.
Variable index TV, variable index Besov regularization terms have been proposed in image analysis,
however, there are no such prior measure in Bayesian theory. So in this paper, we propose a variable index Besov prior measure
which is a Non-Guassian measure.
Based on the variable index Besov prior measure, we build the Bayesian inverse theory.
Then applying our theory to integer and fractional order backward diffusion problems.
Although there are many researches about fractional order backward diffusion problems, we firstly apply Bayesian
inverse theory to this problem which provide an opportunity to quantify the uncertainties for this problem.
\end{abstract}

\maketitle


\section{Introduction}

Partial differential equations are powerful tools for describing physical systems. Using partial differential
equations we can predict the outcome of some measurements. The inverse problem consists of using
the actual result of some measurements to infer the values of the parameters that characterize the system.
Partial differential equations usually have a unique solution, while the inverse problem does not \cite{Tarantola}.
Because of this, we usually need to use some a priori information to compensate the data uncertainties.

Regularization techniques are useful tools to produce a reasonable estimate of quantities of interest based on the data available.
Studies about regularization techniques have a long history dating back to A. N. Tikhonov in 1963 \cite{regularize}.
When studying backward diffusion problems, the researchers found that the classical Tikhonov regularization with standard penalty terms
such as $\|\nabla u\|_{L^{2}}^{2}$ is well known to be always over smoothing.
So total variation (TV) regularization has been proposed in \cite{total_variation}.
TV penalty have gained increasing popularity for it can preserve the important details such as edges of the image.
In 1997, P. Blomgren, T. F. Chan, P. Mulet and C. K. Wong \cite{variable_total} noticed that
TV restoration typically exhibit ``blockiness'', or a ``staircasing'' effect where the restored image comprises
of piecewise flat regions. So they proposed a regularization term as follows
\begin{align*}
\int_{\Omega} |\nabla u|^{p(\nabla u)} dx,
\end{align*}
where $p$ monotonically decreasing from 2, when $|\nabla u| = 0$, to 1, as $|\nabla u|\nearrow \infty$.
Later in 2006, Y. Chen, S. Levine and M. Rao \cite{pde_variable} proposed another kind of variable index TV norm
which has many mathematical structures. In 2014, J. Tiirola \cite{decomposition_variabel} used variable index TV norm
and variable index Besov regularization terms in image decomposition problems.

However, regularization techniques can not give uncertainty analysis. Statistical inversion theory
reformulate inverse problems as problems of statistical inference by means of Bayesian statistics.
Dating back to 1970, Franklin \cite{franklin} formulating PDE inverse problems in terms of Bayes' formula on some Hilbert
space $X$. In this paper Franklin deriving a regularization using Baysian approach, it also state the
relation between regularization techniques and Bayesian approach.
Recently, S. Lasanen \cite{lasanen1,lasanen2,lasanen3,lasanen4} developed fully nonlinear theory.
S. L. Cotter, M. Dashti. J. C. Robinson, A. M. Stuart, K. J. H. Law and J. Voss \cite{inverse_fluid_equation,acta_numerica,MAP_detail}
establish a mathematical framework for a range of inverse problems for functions given noisy observations.
They establish the relation between regularization techniques and Bayesian framework, estimate the error of the
finite dimensional approximate solutions.
Based on this framework, S. L. Cotter, G. O. Roberts, A. M. Stuart and D. White \cite{mcmc_faster}
develop faster MCMC algorithms.

Now, we need to state the relation between TV regularization and Bayesian approach.
In \cite{can_total}, M. Lassas and S. Siltanen find TV regularization is not discretization invariant,
that is to say, Bayesian conditional mean estimates for total variation prior distribution are not edge-preserving with very fine
discretizations of the model space.
In order to overcome this deficiency, M. Lassas, E. Saksman and S. Siltanen \cite{besov_pp} proposed
Besov prior $B_{1,1}^{1}$ which is discretization invariant.
M. Dashti, S. Harris and A. M. Stuart \cite{Besov_prior} studied the Besov prior under the
mathematical framework established in \cite{inverse_fluid_equation}. \
Under this framework Besov prior is discretization invariant naturally for the framework is originally
build on infinite dimensional space.

Considering the Besov and TV regularization techniques and Bayes' inversion theory,
we find that there may no variable index Besov prior theory. As mentioned before
variable index TV and variable index Besov regularization terms have been used in image analysis and
achieved good performance. In this paper, we try to build variable index Besov prior
and generalize the Bayes' inversion theory with this new prior probability measure.
In section \ref{application}, we use our theory to integer order backward diffusion problems and
fractional order backward diffusion problems.

The main contributions of this paper are
\begin{enumerate}
  \item Construct a variable index Besov prior using wavelet characterization of variable index Besov space and prove a
  Fernique-like result \cite{book_SDE} for variable index Besov prior;
  \item Based on variable index Besov prior, we generalize results in \cite{acta_numerica} to build the Bayes' inversion theory.
  Under same conditions for the forward operator, we also prove the convergence of variational problems with variable index
  Besov regularization term;
  \item Although there are many studies about inverse problems for fractional diffusion equations \cite{fractional1,fractional2},
  until now, there are few studies about fractional order backward diffusion problems under Bayes' inversion framework.
  Using our theory, we proved the posterior measure exist and the continuity of the posterior measure with respect to
  the data for integer order backward diffusion problems and fractional order backward diffusion problems.
\end{enumerate}

The content of this paper are organized as follows.
In section \ref{pri}, we state some basic knowledge about variable index space and
prove wavelet characterization of variable index Besov space with periodic domain.
In section \ref{prior_section}, we construct the variable index Besov prior and proved a Fernique-like theorem.
In section \ref{bayesion_section}, we generalize Bayesian inversion theory to our variable index Besov
prior setting.
In section \ref{variational_section}, under same conditions in section \ref{bayesion_section}
for forward problem, we proved the variational problem with variable index Besov regularization term converge.
In section \ref{application}, we used our theory to integer order backward diffusion problem and
fractional order backward diffusion problem.
In the last two sections, we give some technical lemmas and for the reader's convenience we list some useful
theorems and lemmas used in our paper.


\section{Variable order space and wavelet characterization}\label{pri}

In this section, we give a short introduction to space of variable smoothness and integrability on periodic domain, then
we prove a wavelet characterization of variable index Besov space on periodic domain.

\subsection{Modular spaces}\label{modular}
\begin{definition}\label{definition_semimodular}\cite{variable_book}
Let $X$ be a vector space over $\mathbb{R}$ or $\mathbb{C}$. A function $\rho : X \rightarrow [0,\infty]$ is called a semimodular
on $X$ if the following properties hold:
\begin{enumerate}
  \item $\rho(0) = 0$.
  \item $\rho(\lambda f) = \rho(f)$ for all $f \in X$ and $|\lambda| = 1$.
  \item $\rho(\lambda f) = 0$ for all $\lambda > 0$ implies $f = 0$.
  \item $\lambda \mapsto \rho(\lambda f)$ is left continuous on $[0,\infty)$ for every $f \in X$.
\end{enumerate}
\end{definition}
A semimodular $\rho$ is called a modular if

$\,\,\,\,$(5) $\rho(f) = 0$ implies $f = 0$.

A semimodular $\rho$ is called continuous if

$\,\,\,\,$(6) for ever $f \in X$ the mapping $\lambda \mapsto \rho(\lambda f)$ is continuous on $[0,\infty)$.

A semimodular $\rho$ can be additionally qualified by the term convex. This means, as usual, that
\begin{align*}
\rho(\theta f + (1-\theta)g) \leq \theta \rho(f) + (1-\theta)\rho(g),
\end{align*}
for all $f,g \in X$.

Once we have a semimodular in place, we obtain a normed space in a standard way:
\begin{definition}\label{definition_pre_space}\cite{variable_book}
If $\rho$ is a (semi)modular on $X$, then
\begin{align*}
X_{\rho} := \{ x \in X :\, \exists \, \lambda > 0, \, \rho(\lambda x) < \infty \}
\end{align*}
is called a (semi)modular space.
\end{definition}
\begin{theorem}\label{definition_pre_norm}\cite{variable_book}
Let $\rho$ be a convex semimodular on $X$. Then $X_{\rho}$ is a normed space with the Luxemburg norm given by
\begin{align*}
\|x\|_{\rho} := \inf \left\{ \lambda > 0 :\, \rho\left( \frac{1}{\lambda}x \right) \leq 1 \right\}.
\end{align*}
\end{theorem}

\subsection{Spaces of variable integrability}

The variable exponents that we consider are always measurable function on n-dimensional torus $\mathbb{T}^{n}$ with range $[1,\infty)$.
We denote the set of such functions by $\mathcal{P}$.
We denote $p^{+} = \mathrm{esssup}_{x\in \mathbb{T}^{n}} p(x)$ and $p^{-} = \mathrm{essinf}_{x\in \mathbb{T}^{n}} p(x)$.
The function $\varphi_{p}$ is defined as follows:
\begin{align*}
\varphi_{p}(t) =
\left\{
  \begin{array}{ll}
    t^{p}, & \text{if } p \in (0,\infty), \\
    0, &  \text{if } p=\infty \text{ and }t \leq 1,  \\
    \infty, &  \text{if } p=\infty \text{ and }t>1.
  \end{array}
\right.
\end{align*}
The convention $1^{\infty} = 0$ is adopted in order that $\varphi_{p}$ is left continuous.
In what follows we write $t^{p}$ instead of $\varphi_{p}(t)$. The variable exponent modular is defined by
\begin{align*}
\rho_{L^{p(\cdot)}}(f) := \int_{\mathbb{T}^{n}} |f(x)|^{p(x)} dx.
\end{align*}
The variable exponent Lebesgue space $L^{p(\cdot)}$ and its norm $\|f\|_{p(\cdot)}$ are defined by the modular as explained in the previous subsection.

We say that $g : \mathbb{T}^{n} \rightarrow \mathbb{R}$ is locally log-H\"{o}lder continuous, abbreviated $g \in C_{loc}^{log}(\mathbb{T}^{n})$,
if there exists $c> 0$ such that
\begin{align*}
|g(x) - g(y)| \leq \frac{c}{\log(e + 1/|x-y|)}
\end{align*}
for all $x, y \in \mathbb{T}^{n}$. We say that $g$ is globally log-H\"{o}lder continuous, abbreviated $g \in C^{log}$, if it is locally
log-H\"{o}lder continuous and there exists $g_{\infty} \in \mathbb{R}$ such that
\begin{align*}
|g(x) - g_{\infty}| \leq \frac{c}{\log(e + |x|)}
\end{align*}
for all $x \in \mathbb{T}^{n}$. The notation $\mathcal{P}^{log}$ is used for those variable exponents $p \in \mathcal{P}$
with $\frac{1}{p} \in C^{log}$ that is to say $1\leq p^{-} \leq p(x) \leq p^{+} < \infty$ and $\frac{1}{p}$ is globally
log-H\"{o}lder continuous.

\subsection{Variable index Besov space}
Before we introduce variable index Besov space, we need the following definition of mixed Lebesgue-sequence space.
\begin{definition}\label{mixed_Legesgue_sequence}\cite{Almeida_Hasto}
Let $p,q \in \mathcal{P}$. The mixed Lebesgue-sequence space $\ell^{q(\cdot)}(L^{p(\cdot)})$ is defined on
sequences of $L^{p(\cdot)}$-functions by modular
\begin{align*}
\rho_{\ell^{q(\cdot)}(L^{p(\cdot)})}(\{f_{v}\}_{v})
:= \sum_{v} \inf \left\{ \lambda_{v} > 0 \, | \, \rho_{L^{p(\cdot)}}(f_{v}/\lambda^{\frac{1}{q(\cdot)}}_{v}) \leq 1 \right\}.
\end{align*}
\end{definition}

As usual, denote the Fourier transform of a distribution or a function $f$ as $\mathcal{F}(f)$ or $\hat{f}$.
Denote the inverse Fourier transform of a distribution or a function $f$ as $f^{\vee}$.
As in the constant index case, we need the following definition of admissible functions.
\begin{definition}\label{admissible_pair}\cite{Almeida_Hasto}
We say a pair $(\varphi, \Phi)$ is admissible if $\varphi, \Phi \in \mathcal{S}$ satisfy
\begin{itemize}
  \item $\mathrm{supp}\, \hat{\varphi} \subset \{ \xi \in \mathbb{T}^{n} \,:\, 1/2 \leq |\xi| \leq 2 \}$ and
  $|\hat{\varphi}(\xi)| \geq c > 0$ when $3/5 \leq |\xi| \leq 5/3$,
  \item $\mathrm{supp}\, \hat{\Phi} \subset \{ \xi \in \mathbb{T}^{n} \,:\, |\xi| \leq 2 \}$ and $|\hat{\Phi}(\xi)| \geq c > 0$
  when $|\xi| \leq 5/3$.
\end{itemize}
\end{definition}
We set $\varphi_{v}(x) := 2^{vn} \varphi(2^{v}x)$ for $v \in \mathbb{N}$ and $\varphi_{0}(x) := \Phi(x)$.
Denote $\mathcal{S}$ to be the Schwartz function space, $\mathcal{S}'$ to be the tempered distribution that is the dual space of $\mathcal{S}$.
Then the variable index Besov space in our setting can be defined as follows.
\begin{definition}\label{definition_variable_Besov}\cite{Almeida_Hasto}
Let $\varphi_{v}$ be as in Definition \ref{admissible_pair}. For $\alpha : \mathbb{T}^{n} \rightarrow \mathbb{R}$
and $p,q \in \mathcal{P}$, the variable index Besov space $B^{\alpha(\cdot)}_{p(\cdot),q(\cdot)}$ consists of all
distributions $f \in \mathcal{S}'$ such that
\begin{align*}
\|f\|_{B^{\alpha(\cdot)}_{p(\cdot),q(\cdot)}} :=
\left\| \left( 2^{v\alpha(\cdot)}\varphi_{v}*f \right)_{v} \right\|_{\ell^{q(\cdot)}(L^{p(\cdot)})} < \infty.
\end{align*}
\end{definition}
In the case of $p = q$ we use the notation $B^{\alpha(\cdot)}_{q(\cdot)} := B^{\alpha(\cdot)}_{p(\cdot),q(\cdot)}$.
To the Besov space we can also associate the following modular:
\begin{align*}
\rho_{b^{\alpha(\cdot)}_{p(\cdot),q(\cdot)}} := \rho_{\ell^{q(\cdot)}(L^{p(\cdot)})}
( ( 2^{v\alpha(\cdot)}\varphi_{v}*f )_{v} ),
\end{align*}
which can be used to define the norm.
For the reader's convenience, we also list the definition of variable index Triebel-Lizorkin space $F^{\alpha(\cdot)}_{p(\cdot),q(\cdot)}$.
\begin{definition}\label{definition_variable_Triebel}\cite{variable_tribel}
Let $\varphi_{v}$, $v \in \mathbb{N}\cup 0$, be as in Definition \ref{admissible_pair}.
The Triebel-Lizorkin space $F^{\alpha(\cdot)}_{p(\cdot),q(\cdot)}$ is defined to be the space of all distributions $f \in \mathcal{S}'$
with $\|f\|_{F^{\alpha(\cdot)}_{p(\cdot),q(\cdot)}} < \infty$, where
\begin{align*}
\|f\|_{F^{\alpha(\cdot)}_{p(\cdot),q(\cdot)}} :=
\left\| \left\| 2^{v\alpha(\cdot)}\varphi_{v}*f \right\|_{\ell^{q(\cdot)}} \right\|_{L^{p(\cdot)}}.
\end{align*}
\end{definition}
In the case of $p = q$ we use the notation $F^{\alpha(\cdot)}_{q(\cdot)}:=F^{\alpha(\cdot)}_{q(\cdot),q(\cdot)}$.
In the following of this paper, we denote $A \approx B$ equal to $c A \leq B \leq CA$ with
$c,C$ be two constants.

\subsection{Wavelet characterization}

Now, we state some notations for wavelet theory then prove a
wavelet characterization of variable index Besov and Triebel-Lizorkin space on periodic domain.

Let $\psi^{M}$, $\psi^{F}$ be the Meyer or Daubechies wavelets described in Proposition \ref{definition_wavelet} in the Appendix.
Now define
\begin{align*}
G^{0} = \{ F, M \}^{n} \quad \text{and} \quad G^{j} = \{ F,M \}^{n*} \text{ if } j \geq 1,
\end{align*}
where the $*$ indicates that at least one $G_{i}$ of $G = (G_{1},\cdots,G_{n}) \in \{ F,M \}^{n*}$ must be an $M$.
It is clear from the definition that the cardinal number of $\{ F, M \}^{n*}$ is $2^{n} - 1$.
Let $x \in \mathbb{R}^{n}$
\begin{align}\label{high_dimension_wavelet}
\Psi_{Gm}^{j}(x) = 2^{j\frac{n}{2}} \prod_{r = 1}^{n} \psi^{Gr}(2^{j}x_{r} - m_{r}),
\end{align}
where $G \in G^{j}$, $m \in \mathbb{Z}^{n}$ and $j \in \mathbb{N}_{0}$.
Then $\{ \Psi_{Gm}^{j} : j \in \mathbb{N}_{0}, \, G\in G^{j} ,\, m \in \mathbb{Z}^{n} \}$ is an orthomormal basis in $L^{2}(\mathbb{R}^{n})$.

Define
\begin{align*}
\psi^{M}_{j,k}(x) := 2^{\frac{j}{2}} \psi^{M}(2^{j}x - k) \quad \psi^{F}_{j,k}(x) := 2^{\frac{j}{2}} \psi^{F}(2^{j}x - k)
\end{align*}
where $k \in \mathbb{Z}$.
Then, we define
\begin{align}\label{periodic_mother}
\tilde{\psi}^{M}_{j,k}(x) := \sum_{\ell \in \mathbb{Z}}\psi^{M}_{j,k}(x + \ell)
= 2^{j\frac{n}{2}}\sum_{\ell \in \mathbb{Z}} \psi^{M}(2^{j}(x+\ell)-k)
\end{align}
and
\begin{align}\label{periodic_farther}
\tilde{\psi}^{F}_{j,k}(x) := \sum_{\ell \in \mathbb{Z}}\psi^{F}_{j,k}(x + \ell)
= 2^{j\frac{n}{2}}\sum_{\ell \in \mathbb{Z}} \psi^{F}(2^{j}(x+\ell)-k).
\end{align}
Obviously, $\tilde{\psi}_{j,k}$, $\tilde{\phi}_{j,k}$ are $1$-periodic functions belongs to $L^{1}([0,1])$.

Define
\begin{align}\label{periodic_define_wavelet}
\tilde{\Psi}_{Gm}^{j}(x) = 2^{j\frac{n}{2}} \prod_{r = 1}^{n} \tilde{\psi}^{Gr}(2^{j}x_{r} - m_{r}),
\end{align}
By Proposition \ref{definition_wavelet}, we know that $\psi^{M}$, $\psi^{F}$ included in the
functions with radial decreasing $L^{1}$-majorants, that is
\begin{align*}
|\psi^{M}(x)| \leq R_{1}(|x|) \quad  |\psi^{F}(x)| \leq R_{2}(|x|),
\end{align*}
where $R_{1}$ and $R_{2}$ are bounded decreasing functions belongs to $L^{1}([0,\infty))$.
Now, we can use Theorem 5.9 in \cite{Hernandez_Weiss} to find that
$\{ \tilde{\Psi}_{Gm}^{j} : j \in \mathbb{N} \cup \{0\}, \, G\in G^{j} ,\, m \in \mathbb{M}_{j} \}$
with $\mathbb{M}_{j} = \{ m: m = 0,1,2,\cdots,2^j-1\}$ is an orthomormal basis in $L^{2}(\mathbb{T}^{n})$.
At this point, considering Corollary 5 in \cite{Kermpka} and Definition \ref{atom_periodic} in the Appendix,
we easily obtain the following theorem for wavelet
characterization of variable index Besov and Triebel-Lizorkin space on periodic domain.

\begin{theorem}\label{kong_jian_wavelet}
Let $s(\cdot) \in L^{\infty} \cap C_{loc}^{log}(\mathbb{T}^{n})$ and $p(\cdot) \in \mathcal{P}^{log}(\mathbb{T}^{n})$.
The symbol $A$ stands for $B$ or $F$ and so does a symbolize $b$ or $f$, respectively.

(i) Let $0 < q \leq \infty$ ($p^{+} < \infty$ in the $F$-case) and
\begin{align*}
k > \max (\sigma_{p} - s^{-}, s^{+}) \quad (\sigma_{p,q} \text{ in the }F\text{-case}),
\end{align*}
where $\sigma_{p} = n \left( \frac{1}{\min(1, p^{-})} - 1 \right)$
and $\sigma_{p,q} = n \left( \frac{1}{\min(1,p^{-},q)} - 1 \right)$.
Then $f \in \mathcal{S}'(\mathbb{T}^{n})$ belongs to $A_{p(\cdot), q}^{s(\cdot)}$ if, and only if,
it can be represented as
\begin{align}\label{periodic_expansion_function}
\begin{split}
f = \sum_{j = 0}^{\infty}\sum_{G \in G^{j}}\sum_{m \in \mathbb{M}_{j}} \lambda_{Gm}^{j} 2^{-j\frac{n}{2}}\tilde{\Psi}_{Gm}^{j}
\quad \text{with }\lambda \in \tilde{a}_{p(\cdot),q}^{s(\cdot)},
\end{split}
\end{align}
with $\mathbb{M}_{j} = \{ m: m = 0,1,2,\cdots,2^j-1 \}$ and
the series expansion (\ref{periodic_expansion_function}) is unconditional convergence in $\mathcal{S}'(\mathbb{R}^{n})$
and in any space $A_{p(\cdot),q}^{\sigma(\cdot)}(\mathbb{T}^{n})$, where $\sigma(x) < s(x)$ with $\inf (s(x) - \sigma(x)) > 0$
and $\sigma(x)/s(x) \rightarrow 0$ for $|x| \rightarrow \infty$. The representation (\ref{periodic_expansion_function}) is unique, we have
\begin{align*}
\lambda_{Gm}^{j} = \lambda_{Gm}^{j}(f) = 2^{j\frac{n}{2}} <f, \tilde{\Psi}_{Gm}^{j}>
\end{align*}
and
\begin{align*}
I: f \mapsto \{ 2^{j\frac{n}{2}} <f, \tilde{\Psi}_{Gm}^{j}> \}
\end{align*}
is an isomorphic map from $A_{p(\cdot),q}^{s(\cdot)}(\mathbb{T}^{n})$ onto $\tilde{a}_{p(\cdot),q}^{s(\cdot)}$.
Moreover, if in addition $\max(p^{+}, q) < \infty$, then $\{ \tilde{\Psi}_{Gm}^{j} \}_{j \in \mathbb{N}_{0}, G\in G^{j}, m\in \mathbb{M}_{j}}$
is an unconditional basis in $A_{p(\cdot),q}^{s(\cdot)}(\mathbb{T}^{n})$.

(ii) Let $q(\cdot) \in \mathcal{P}^{log}(\mathbb{T}^{n})$ with $0 < p^{-} \leq p^{+} < \infty$, $0 < q^{-} \leq q^{+} \leq \infty$
and let
\begin{align*}
k > \max(\sigma_{p,q} - s^{-}, s^{+})
\end{align*}
with $\sigma_{p,q} = n \left( \frac{1}{\min(1,p^{-},q^{-})} - 1 \right)$.
Then $f \in \mathcal{S}'(\mathbb{T}^{n})$ belongs to $F_{p(\cdot),q(\cdot)}^{s(\cdot)}(\mathbb{T}^{n})$
if, and only if, it can be represented as (\ref{periodic_expansion_function})
with $\lambda \in \tilde{f}_{p(\cdot),q(\cdot)}^{s(\cdot)}(\mathbb{T}^{n})$,
with unconditional convergence in $\mathcal{S}'(\mathbb{T}^{n})$ and in $\lambda \in \tilde{f}_{p(\cdot),q(\cdot)}^{s(\cdot)}(\mathbb{T}^{n})$.
The representation (\ref{periodic_expansion_function}) is unique, we have
\begin{align*}
\lambda_{Gm}^{j} = \lambda_{Gm}^{j}(f) = 2^{j\frac{n}{2}}<f, \tilde{\Psi}_{Gm}^{j}>
\end{align*}
and
\begin{align*}
I : f \mapsto \{ 2^{j\frac{n}{2}} <f, \tilde{\Psi}_{Gm}^{j}> \}
\end{align*}
is an isomorphic map from $F_{p(\cdot),q(\cdot)}^{s(\cdot)}(\mathbb{T}^{n})$ onto $\tilde{f}_{p(\cdot),q(\cdot)}^{s(\cdot)}(\mathbb{T}^{n})$.
\end{theorem}

At the end of this section, we need to introduce the following notation which used frequently in the sequel.
\begin{align}
\rho_{B_{q(\cdot)}^{s(\cdot)}}(u) =
\int_{\mathbb{T}^{n}}
\sum_{j = 0}^{\infty} \sum_{G \in G^{j}} \sum_{m \in \mathbb{M}_{j}} 2^{jq(x)s(2^{-j}m)} |\lambda_{Gm}^{j}|^{q(x)}
\chi_{jm}(x) dx,
\end{align}
where $\lambda_{Gm}^{j}$ are defined as in Theorem \ref{kong_jian_wavelet}.


\section{Variable order Besov prior}\label{prior_section}

For the reader's convenience, let us recall the general setting stated in \cite{Dashti_Stuart} for our purpose.
Denote $J$ to be an index set, let $\{\phi_{j}\}_{j \in J}$ denote an infinite sequence in the Banach space $X$, with norm $\|\cdot\|$,
of $\mathbb{R}$-valued functions defined on a domain $D$. In the following, for simplicity, we assume $D = \mathbb{T}^{n}$ to be
the $n$-dimensional torus. We normalize these functions so that $\|\phi_{j}\| = 1$ for $j \in J$.
We also introduce another element $m_{0} \in X$, not necessarily normalized to $1$.
Define the function $u$ by
\begin{align}\label{random_function_general}
u = m_{0} + \sum_{j \in J} u_{j}\phi_{j}
\end{align}
By randomizing $u := \{ u_{j} \}_{j \in J}$ we create real-valued random functions on $D$.
(The extension to $\mathbb{R}^{n}$-valued random functions is straightforward, but omitted for brevity.)
We now define the deterministic sequence $\gamma = \{ \gamma_{j} \}_{j \in J}$ and the i.i.d. random sequence
$\xi = \{ \xi_{j} \}_{j \in J}$, and set $u_{j} = \gamma_{j}\xi_{j}$. We assume that $\xi$ is centered, i.e. that
$\mathbb{E}(\xi_{1}) = 0$. Formally we see that the average value of $u$ is then $m_{0}$ so that this
element of $X$ should be thought of as the mean function.

In the following, we take $X$ to be the Hilbert space
\begin{align*}
X := L^{2}(\mathbb{T}^{n}) = \left\{ u \, : \, \mathbb{T}^{n} \rightarrow \mathbb{R} \, : \,
\int_{\mathbb{T}^{n}} |u(x)|^{2} dx < \infty \right\}
\end{align*}
of real valued periodic functions in dimension $n \leq 3$ with inner product and norm denoted by $<\cdot, \cdot>$
and $\|\cdot\|$ respectively.
We then set $m_{0} = 0$ and let
\begin{align*}
J = \{ j = 0,1,\cdots, \, G^{j}, \,  \mathbb{M}_{j}\}
\end{align*}
appeared in the Theorem \ref{kong_jian_wavelet}
which is an orthonormal basis for $X$.
Consequently, for any $u \in X$, we have
\begin{align}\label{var1_expansion}
\begin{split}
u(x) = \sum_{j = 0}^{\infty}\sum_{G \in G^{j}}\sum_{m \in \mathbb{M}_{j}} u_{Gm}^{j} \tilde{\Psi}_{Gm}^{j} \quad \text{with }
u_{Gm}^{j} = <u, \tilde{\Psi}_{Gm}^{j}>,
\end{split}
\end{align}
where $\tilde{\Psi}_{Gm}^{j}$ is the wavelet basis stated in the Theorem \ref{kong_jian_wavelet}.
Given a function $u : \mathbb{T}^{n} \rightarrow \mathbb{R}$ and the $\{ u_{Gm}^{j} \}$ as defined in (\ref{var1_expansion})
we define the Banach space $B_{q(\cdot)}^{t(\cdot)}$ by
\begin{align}\label{var_define_space}
\begin{split}
B_{q(\cdot)}^{t(\cdot)} = \left\{ u:\mathbb{T}^{n} \rightarrow \mathbb{R}\,:\, \|u\|_{F_{q(\cdot),q(\cdot)}^{t(\cdot)}}< \infty \right\}
\end{split}
\end{align}
with
\begin{align}\label{var_define_eql_space}
\|u\|_{F_{q(\cdot),q(\cdot)}^{t(\cdot)}} =
\left\|\left( \sum_{k=0}^{\infty}\left| \left( \varphi_{k}\hat{u} \right)^{\vee}
2^{kt(x)} \right|^{q(\cdot)} \right)^{1/q(\cdot)} \right\|_{L^{q(\cdot)}(\mathbb{T}^{n})},
\end{align}
where $\varphi_{j}(x) = \varphi(2^{-j}x)$ and $\varphi(\cdot)$ is a smooth decompositions of unity \cite{DanchinBook,Triebel2008}.
By Proposition 5.4 in \cite{Almeida_Hasto}, we can define
the space $B_{p(\cdot),p(\cdot)}^{s(\cdot)}$ appropriately which is equivalent to $F_{p(\cdot),p(\cdot)}^{s(\cdot)}$
if $s \in L^{\infty}$. So here, we do not need to develop a full theory for the space $B_{p(\cdot),p(\cdot)}^{s(\cdot)}$,
and just understand it as $F_{p(\cdot),p(\cdot)}^{s(\cdot)}$.
Hence, our space $B_{q(\cdot)}^{t(\cdot)}$ defined in (\ref{var_define_space}) just
the usual variable index Besov space with $p(\cdot) = q(\cdot)$.
(Although the space defined in \cite{Almeida_Hasto} is in the whole space $\mathbb{R}^{n}$, it can be
adapted to periodic case $\mathbb{T}^{n}$.)

Now, as in the general setting, we assume that $u_{Gm}^{j} = \gamma_{Gm}^{j} \xi_{Gm}^{j}$ where
$\xi = \{ \xi_{Gm}^{j} \}_{j=1,2,\cdots,\infty, G \in G^{j}, m \in \mathbb{M}_{j}}$ is an i.i.d. sequence and
$\gamma = \{ \gamma_{Gm}^{j} \}_{j=1,2,\cdots,\infty, G \in G^{j}, m \in \mathbb{M}_{j}}$ is deterministic.
Here we assume that $\xi_{Gm}^{j}$ is draw from the centred measure on $\mathbb{R}$ with density proportional
to $\exp\left( -\frac{1}{2} \int_{\mathbb{T}^{n}} |x|^{q(y)} \kappa(dy) \right)$ for some $1 \leq q^{-} \leq q^{+} < \infty$,
where $\kappa(\cdot)$ is a probability measure. We refer to the measure with above density as a generalized $q(\cdot)$-exponential
distribution.
Noting that if $q$ is constant, it is just a $q$-exponential distribution \cite{Dashti_Stuart}.
Hence, our generalized $q(\cdot)$-exponential distribution is a natural extension of $q$-exponential distribution and include
Gaussian, Laplace distribution as special case. For $s(x) \geq s^{-} > 0$ and $\delta > 0$ we define
\begin{align}\label{coefficient_de}
\begin{split}
\gamma_{Gm}^{j} = 2^{-j (s(2^{-j}m) + n/2 - n/q^{+})} \left( \frac{1}{\delta} \right)^{1/q^{+}}.
\end{split}
\end{align}

We now prove convergence of the series
\begin{align}\label{cut_off_series}
\begin{split}
u^{N} = \sum_{j = 0}^{N}\sum_{G \in G^{j}}\sum_{m \in \mathbb{M}_{j}} u_{Gm}^{j} \tilde{\Psi}_{Gm}^{j}, \quad
u_{Gm}^{j} = \gamma_{Gm}^{j}\xi_{Gm}^{j}
\end{split}
\end{align}
to the limit function
\begin{align}\label{full_series}
\begin{split}
u = \sum_{j = 0}^{\infty}\sum_{G \in G^{j}}\sum_{m \in \mathbb{M}_{j}} u_{Gm}^{j} \tilde{\Psi}_{Gm}^{j}, \quad
u_{Gm}^{j} = \gamma_{Gm}^{j}\xi_{Gm}^{j},
\end{split}
\end{align}
in an appropriate space. In order to understand the sequence of functions $\{ u^{N} \}$ fully, we introduce the following
function space:
\begin{align*}
L_{\mathbb{P}}^{q(\cdot)}(\Omega; B^{t(\cdot)}_{q(\cdot)}) :=
\left\{ u: \mathbb{T}^{n} \times \Omega \rightarrow \mathbb{R} \, : \,  \exists \, \lambda > 0, \, \rho^{E}_{B^{t(\cdot)}_{q(\cdot)}}
(\lambda u) < \infty  \right\}
\end{align*}
where
\begin{align*}
u_{k} = \left( \varphi_{k}\hat{u} \right)^{\vee}
\end{align*}
with $\varphi_{k}$ defined as in (\ref{var_define_eql_space}), and
\begin{align}\label{siminorm}
\begin{split}
\rho^{E}_{B^{t(\cdot)}_{q(\cdot)}}(u) & = \sum_{k=0}^{\infty} \inf \left\{ \lambda_{k} > 0\,:\,
\int_{\Omega} \rho_{L^{q(\cdot)}}(u_{k}\lambda_{k}^{-1/q(\cdot)}2^{kt(\cdot)}) \mathbb{P}(d\omega) \leq 1 \right\}  \\
& = \int_{\Omega} \int_{\mathbb{T}^{n}} \sum_{k = 0}^{\infty} 2^{k t(x) q(x)} |u_{k}(x,\omega)|^{q(x)} dx \mathbb{P}(d\omega).
\end{split}
\end{align}
As mentioned in section \ref{modular}, if $\rho^{E}_{B^{t(\cdot)}_{q(\cdot)}}$ is
a convex semimodular on $L_{\mathbb{P}}^{q(\cdot)}\left(\Omega; B^{t(\cdot)}_{q(\cdot)}\right)$.
Then $L_{\mathbb{P}}^{q(\cdot)}\left(\Omega; B^{t(\cdot)}_{q(\cdot)}\right)$ is a normed space with the Luxemburg norm given by
\begin{align}\label{var_define_norm}
\begin{split}
\|u\|_{L_{\mathbb{P}}^{q(\cdot)}(\Omega; B^{t(\cdot)}_{q(\cdot)})}
= \inf \left\{ \mu > 0 \, : \, \rho^{E}_{B^{t(\cdot)}_{q(\cdot)}}\left(\left( \frac{1}{\mu} \right) u\right) \leq 1 \right\}.
\end{split}
\end{align}
In order to keep the fluency of our statement,
we list the proof of $\rho^{E}_{B^{t(\cdot)}_{q(\cdot)}}$ is a convex semimodular in section \ref{yin_li}.
Now, we clarify the relation for our space $L_{\mathbb{P}}^{q(\cdot)}(\Omega; B^{t(\cdot)}_{q(\cdot)})$ with
the usual constant $q, t$ space $L_{\mathbb{P}}^{q}(\Omega; B^{t}_{q})$ used in \cite{Besov_prior}.
Let $q$, $t$ in (\ref{siminorm}) to be constants, we have
\begin{align*}
\rho^{E}_{B^{t}_{q}}(u) & = \int_{\Omega}\sum_{k=0}^{\infty}2^{k t q} \int_{\mathbb{T}^{n}} |u_{k}|^{q} dx \mathbb{P}(d\omega) \\
& = \mathbb{E}(\|u\|_{B^{t}_{q,q}}^{q}).
\end{align*}
Hence, our variable space indeed is a natural generalization of the usual space $L_{\mathbb{P}}^{q}(\Omega; B^{t}_{q})$.
Define
\begin{align}\label{xu_lie}
\begin{split}
\tilde{b}_{q(\cdot)}^{E \, s(\cdot)} := \left\{ \lambda = \{ \lambda_{Gm}^{j} \}_{j \in \mathbb{N}_{0}, G \in G^{j}, m \in \mathbb{M}_{j}}
: \, \|\lambda\|_{\tilde{b}_{q(\cdot)}^{E \, s(\cdot)}} < \infty \right\},
\end{split}
\end{align}
where
\begin{align*}
\|\lambda\|_{\tilde{b}_{q(\cdot)}^{E \, s(\cdot)}} = \left\| \left(
\sum_{j = 0}^{\infty} \sum_{G \in G^{j}} \sum_{m \in \mathbb{M}_{j}} 2^{jqs(2^{-j}m)} \mathbb{E}\left(|\lambda_{Gm}^{j}|^{q(\cdot)}\right)
\chi_{jm}(\cdot) \right)^{1/q(\cdot)} \right\|_{L^{q(\cdot)}(\mathbb{T}^{n})},
\end{align*}
and
\begin{align*}
\mathbb{E}\left(|\lambda_{Gm}^{j}|^{q(x)}\right) = \int_{\mathbb{T}^{n}}|\lambda_{Gm}^{j}(\omega)|^{q(x)} \mathbb{P}(d\omega).
\end{align*}
With these definitions, before going to our one main results in this section, we need the following Lemma, which is proved in the Section \ref{yin_li}.
\begin{lemma}\label{wavelet_characterization_measure}
Let $s(\cdot) \in L^{\infty} \cap C_{loc}^{log}(\mathbb{T}^{n})$ and $q(\cdot) \in \mathcal{P}^{log}(\mathbb{T}^{n})$.
Let
\begin{align*}
k > \max (\sigma_{q} - s^{-}, s^{+}),
\end{align*}
where $\sigma_{q} = n \left( \frac{1}{\min(1, q^{-})} - 1 \right)$.
Then $f \in \mathcal{S}'(\mathbb{T}^{n})$ belongs to $L_{\mathbb{P}}^{q(\cdot)}(\Omega; B^{s(\cdot)}_{q(\cdot)})$ if, and only if,
it can be represented as
\begin{align}\label{measure_periodic_expansion_function}
\begin{split}
f = \sum_{j = 0}^{\infty}\sum_{G \in G^{j}}\sum_{m \in \mathbb{M}_{j}} \lambda_{Gm}^{j} 2^{-j\frac{n}{2}}\tilde{\Psi}_{Gm}^{j}
\quad \text{with }\lambda \in \tilde{b}_{q(\cdot)}^{E \, s(\cdot)},
\end{split}
\end{align}
with $\mathbb{M}_{j} = \{ m: m = 0,1,2,\cdots,2^j-1 \}$ and
the series expansion (\ref{measure_periodic_expansion_function}) is unconditional convergence in $\mathcal{S}'(\mathbb{R}^{n})$.
The representation (\ref{measure_periodic_expansion_function}) is unique, we have
\begin{align*}
\lambda_{Gm}^{j} = \lambda_{Gm}^{j}(f) = 2^{j\frac{n}{2}} <f, \tilde{\Psi}_{Gm}^{j}>
\end{align*}
and
\begin{align*}
I: f \mapsto \{ 2^{j\frac{n}{2}} <f, \tilde{\Psi}_{Gm}^{j}> \}
\end{align*}
is an isomorphic map from $L_{\mathbb{P}}^{q(\cdot)}(\Omega; B^{s(\cdot)}_{q(\cdot)})$ onto $\tilde{b}_{q(\cdot)}^{E \, s(\cdot)}$.
\end{lemma}

Now we can prove the following theorem, which gives a sufficient condition, on $t(\cdot)$, for
existence of the limiting random function.
\begin{theorem}\label{convergence_series}
For $t,s \in C^{log}_{loc}(\mathbb{T}^{n}) \cap L^{\infty}(\mathbb{T}^{n})$, $q \in \mathcal{P}^{log}(\mathcal{T}^{n})$ and
$$\sup_{x \in \mathbb{T}^{n}} \left(t(x) - s(x) + \frac{n}{q^{+}}\right) < 0$$
the sequence of functions $\{ u^{N} \}_{N=1}^{\infty}$,
given by (\ref{cut_off_series}) and (\ref{coefficient_de}) with $\xi_{Gm}^{j}$ draw from a centered generalized $q(\cdot)$-exponential
distribution, is Cauchy in the Banach space $L_{\mathbb{P}}^{q(\cdot)}\left(\Omega; B^{t(\cdot)}_{q(\cdot)}\right)$.
Thus the infinite series (\ref{full_series}) exists as an a limit in space $L_{\mathbb{P}}^{q(\cdot)}\left(\Omega; B^{t(\cdot)}_{q(\cdot)}\right)$
for all $\sup_{x \in \mathbb{T}^{n}} \left(t(x) - s(x) + \frac{n}{q^{+}}\right) < 0$.
\end{theorem}
\begin{proof}
By (\ref{full_series}) and the Lemma \ref{wavelet_characterization_measure}, we obtain
\begin{align}\label{dengjiameasure}
\|u\|_{L_{\mathbb{P}}^{q(\cdot)}(\Omega; B^{t(\cdot)}_{q(\cdot)})} \approx
\|\{ 2^{j\frac{n}{2}} u_{Gm}^{j} \}\|_{\tilde{b}_{q(\cdot)}^{E \, s(\cdot)}}
\end{align}
For $M > N$, every $\lambda > 0$, we have the following estimate
\begin{align}\label{guji}
\begin{split}
 & \int_{\mathbb{T}^{n}}
\sum_{j = N+1}^{M} \sum_{G \in G^{j}} \sum_{m \in \mathbb{M}_{j}} \lambda^{q(x)} 2^{jq(x)(t(2^{-j}m) + n/2)}
 |\gamma_{Gm}^{j}|^{q(x)} \mathbb{E}\left(|\xi_{Gm}^{j}|^{q(x)}\right) \chi_{jm}(x) dx \\
\leq & C \max(\lambda^{q^{+}}, \lambda^{q^{-}}) \delta^{-1} \int_{\mathbb{T}^{n}} \sum_{j = N+1}^{M} \sum_{G\in G^{j}} \sum_{m\in \mathbb{M}_{j}}
2^{jq(x)(t(2^{-j}m)-s(2^{-j}m)+n/q^{+})} \chi_{jm}(x) dx    \\
\leq & C \max(\lambda^{q^{+}}, \lambda^{q^{-}}) \delta^{-1}\sum_{j = N+1}^{M}\sum_{m \in \mathbb{M}_{j}} 2^{jq^{-}(t(2^{-j}m)-s(2^{-j}m)+n/q^{+})}
2^{-jn}
\end{split}
\end{align}
where we used
\begin{align*}
\mathbb{E}\left(|\xi_{Gm}^{j}|^{q(x)}\right) & \leq C \int_{\mathbb{R}\cap \{|\xi| > 1\}} |\xi|^{q^{+}}
\exp \left( -\frac{1}{2} \int_{\mathbb{T}^{n}} |\xi|^{q(x)} \kappa(dx) d\xi \right)     \\
& \quad\quad\quad\quad + C \int_{\mathbb{R}\cap \{|\xi| \leq 1\}} |\xi|^{q^{-}}
\exp \left( -\frac{1}{2} \int_{\mathbb{T}^{n}} |\xi|^{q(x)} \kappa(dx) d\xi \right) \\
& \leq C \int_{\mathbb{R}\cap \{|\xi| > 1\}} |\xi|^{q^{+}} \exp\left( -\frac{1}{2} |\xi|^{q^{-}} d\xi \right) \\
& \quad\quad\quad\quad + C \int_{\mathbb{R}\cap \{|\xi| \leq 1\}} |\xi|^{q^{-}} \exp\left( -\frac{1}{2} |\xi|^{q^{+}} d\xi \right) < \infty.
\end{align*}
The sum on the last line of (\ref{guji}) tends to $0$ as $N \rightarrow \infty$, provided
$$\sup_{x \in \mathbb{T}^{n}} \left(t(x) - s(x) + \frac{n}{q^{+}}\right) < 0.$$
Hence, by Lemma 2.1.9. in \cite{variable_book}, we obtain
\begin{align*}
\lim_{N \rightarrow \infty} \|\{ 2^{j\frac{n}{2}} (u_{Gm}^{N \, j} - u_{Gm}^{M \, j}) \}\|_{\tilde{b}_{q(\cdot)}^{E \, s(\cdot)}} = 0.
\end{align*}
Finally, by (\ref{dengjiameasure}), we complete the proof.
\end{proof}
\begin{remark}\label{intuitive_density}
Here we give an intuitive meaning for the random series which we defined in (\ref{full_series}).
Assume the probability measure $\kappa(\cdot)$ in the centered generalized $q(\cdot)$-exponential distribution is
an uniform measure that is
$\kappa(dx) = dx$ in $\mathbb{T}^{n}$.
Since $\tilde{\Psi}_{Gm}^{j}$ is an orthonormal basis and
\begin{align}\label{intuitive_full_series}
\begin{split}
u = \sum_{j = 0}^{\infty}\sum_{G \in G^{j}}\sum_{m \in \mathbb{M}_{j}} u_{Gm}^{j} \tilde{\Psi}_{Gm}^{j}
\end{split}
\end{align}
with $u_{Gm}^{j} = 2^{-j (s(2^{-j}m) + n/2 - n/q^{+})} \left( \frac{1}{\delta} \right)^{1/q^{+}} \xi_{Gm}^{j}$, denote
$\lambda_{Gm}^{j} = 2^{j\frac{n}{2}}u_{Gm}^{j}$, we have
\begin{align*}
& \prod_{j=0}^{\infty}\prod_{G\in G^{j}}\prod_{m\in \mathbb{M}_{j}}
\exp\left( -\frac{1}{2}\int_{\mathbb{T}^{n}} |\xi_{Gm}^{j}|^{q(x)}dx \right)    \\
= & \prod_{j=0}^{\infty}\prod_{G\in G^{j}}\prod_{m\in \mathbb{M}_{j}} \exp \left( -\frac{1}{2}
\int_{\mathbb{T}^{n}} |\gamma_{Gm}^{j}|^{-q(x)}|u_{Gm}^{j}|^{q(x)} dx \right)   \\
= & \prod_{j=0}^{\infty}\prod_{G\in G^{j}}\prod_{m\in \mathbb{M}_{j}} \exp\left( -\frac{1}{2}\int_{\mathbb{T}^{n}}
\delta^{\frac{q(x)}{q^{+}}} 2^{jq(x)\left( s(2^{-jm}) - \frac{n}{q^{+}} \right)} |\lambda_{Gm}^{j}|^{q(x)}dx \right)    \\
\leq & \prod_{j=0}^{\infty}\prod_{G\in G^{j}}\prod_{m\in \mathbb{M}_{j}} \exp\left( -\frac{1}{2}\int_{\mathbb{T}^{n}}
\delta^{\frac{q(x)}{q^{+}}} 2^{jq(x)\left( s(2^{-jm}) - \frac{n}{q^{+}} \right)} |\lambda_{Gm}^{j}|^{q(x)}2^{jn}\chi_{jm}(x) dx \right)    \\
\leq & \exp\left( -\frac{1}{2}\min\{ \delta^{\frac{q^{-}}{q^{+}}}, \delta \}\rho_{B_{q(\cdot)}^{s(\cdot)}}(u) \right).
\end{align*}
Thus, informally, the Lebesgue density of $u$ can be controlled by a Lebesgue density proportional to
$\exp\left( -\frac{1}{2}\min\{ \delta^{\frac{q^{-}}{q^{+}}}, \delta \} \rho_{B_{q(\cdot)}^{s(\cdot)}}(u) \right)$.
Since $\rho_{B_{q(\cdot)}^{s(\cdot)}}(u)$ related to the space $B_{q(\cdot)}^{s(\cdot)}(\mathbb{T}^{n})$
 (Theorem \ref{kong_jian_wavelet}),
and the space $B_{q(\cdot)}^{s(\cdot)}(\mathbb{T}^{n})$ is a generalization of constant index space, we may guess
the Lebesgue density of $u$ related to a Lebesgue density similar to
$\exp\left( -\frac{1}{2}\min\{ \delta^{\frac{q^{-}}{q^{+}}}, \delta \} \rho_{B_{q(\cdot)}^{s(\cdot)}}(u) \right)$.
At least, if we choose $q$ as constant and just $s$ is a function, we will have an
equality which informally means the Lebesgue density of $u$ is proportional to
$\exp\left( -\frac{1}{2}\delta  \rho_{B_{q}^{s(\cdot)}}(u) \right)$.
Hence, the probability measure we defined may be related to the space $B_{q(\cdot)}^{s(\cdot)}(\mathbb{T}^{n})$.
So we may say that $u$ is distributed according to an $B_{q(\cdot)}^{s(\cdot)}(\mathbb{T}^{n})$ measure with parameter $\delta$, or,
briefly, a $(\delta, B_{q(\cdot)}^{s(\cdot)}(\mathbb{T}^{n}))$ measure.
\end{remark}


\begin{theorem}\label{convergence_series_if_only_if}
Assume that $u$ is given by (\ref{full_series}) and (\ref{coefficient_de}) with $\xi_{Gm}^{j}$ for every $\{j,G,m\}$ draw from a
centered generalized $q(\cdot)$-exponential distribution with $\kappa(dx) = dx$ that is to say $\kappa(\cdot)$ is an uniform probability measure.
In other words, $u$ be distributed according to a $(\delta, B^{s(\cdot)}_{q(\cdot)})$ measure.
In addition, we assume $t,s\in C^{log}_{loc}(\mathbb{T}^{n})\cap L^{\infty}(\mathbb{T}^{n})$, $q\in \mathcal{P}^{log}(\mathbb{T}^{n})$
and $t(x) - s(x) + \frac{n}{q^{+}} \neq 0$ for every
$x \in \mathbb{T}^{n}$.
Then the following are equivalent:
\begin{enumerate}
  \item $\rho_{B_{q(\cdot)}^{t(\cdot)}}(u) < \infty \quad \mathbb{P}\text{-a.s.}$;
  \item $\mathbb{E}\left( \exp\left( \alpha \rho_{B_{q(\cdot)}^{t(\cdot)}}(u) \right) \right)<\infty$ for any $\alpha \in [0, \delta/2)$;
  \item $\sup_{x\in \mathbb{T}^{n}} \left(t(x)-s(x)+\frac{n}{q^{+}}\right) < 0$.
\end{enumerate}
\end{theorem}
\begin{proof}
(3) $\Rightarrow$ (2).

Since $\sup_{x\in \mathbb{T}^{n}} \left(t(x)-s(x)+\frac{n}{q^{+}}\right) < 0$, there exists a negative constant $\beta < 0$
such that $\sup_{x\in \mathbb{T}^{n}} \left(t(x)-s(x)+\frac{n}{q^{+}}\right) \leq \beta < 0$.
Let $K$ to be a large enough positive constant, then we have
\begin{align*}
& \mathbb{E}\left( \exp\left( \alpha \rho_{B_{q(\cdot)}^{t(\cdot)}}(u) \right) \right) \\
= &
\prod_{j=0}^{\infty}\prod_{G \in G^{j}} \mathbb{E}\left( \exp( \alpha\delta^{-1}\sum_{m \in \mathbb{M}_{j}}
\int_{\mathbb{T}^{n}} 2^{jq(x)(t(2^{-jm})-s(2^{-jm})+n/q^{+})}|\xi_{Gm}^{j}|^{q(x)}\chi_{jm}(x)dx ) \right) \\
\leq &
\prod_{j=0}^{\infty}\prod_{G\in G^{j}} \mathbb{E}\left( \exp\left( \alpha\delta^{-1}\sum_{m\in \mathbb{M}_{j}}
\int_{\mathbb{T}^{n}} 2^{jq(x)\beta} |\xi_{Gm}^{j}|^{q(x)} \chi_{jm}(x) dx \right) \right)  \\
\leq &
\prod_{j=0}^{\infty}\prod_{G \in G^{j}}
\frac{\int_{\mathbb{R}}\exp\left( \alpha\delta^{-1}\int_{\mathbb{T}^{n}}2^{jq(x)\beta}|\xi|^{q(x)}dx
-\frac{1}{2}\int_{\mathbb{T}^{n}}|\xi|^{q(x)}dx \right)d\xi}{\int_{\mathbb{R}}\exp\left( -\frac{1}{2}
\int_{\mathbb{T}^{n}} |\xi|^{q(x)} dx \right)d\xi} \\
\leq &
\left(\prod_{j=0}^{\infty}
\frac{\int_{\mathbb{R}}\exp\left( \alpha\delta^{-1}\int_{\mathbb{T}^{n}}2^{jq(x)\beta}|\xi|^{q(x)}dx
-\frac{1}{2}\int_{\mathbb{T}^{n}}|\xi|^{q(x)}dx \right)d\xi}{\int_{\mathbb{R}}\exp\left( -\frac{1}{2}
\int_{\mathbb{T}^{n}} |\xi|^{q(x)} dx \right)d\xi} \right)^{K}.
\end{align*}
If we want to prove the above infinite product converge, we only need to prove the following
summation converge \cite{infinite_product}.
\begin{align}\label{var_product_infinite}
\begin{split}
& \sum_{j=0}^{\infty} \int_{\mathbb{R}} \left( \exp\left( \alpha\delta^{-1}\int_{\mathbb{T}^{n}} 2^{jq(x)\beta}|\xi|^{q(x)}dx \right) - 1 \right)
\exp\left( -\frac{1}{2}\int_{\mathbb{T}^{n}} |\xi|^{q(x)} dx \right) d\xi   \\
\leq & \int_{\mathbb{R}} \sum_{j=0}^{\infty}\sum_{k=1}^{\infty}\frac{(\alpha\delta^{-1})^{k}}{k!}\left(
\int_{\mathbb{T}^{n}} 2^{jq(x)\beta}|\xi|^{q(x)}dx \right)^{k} \exp\left( -\frac{1}{2}\int_{\mathbb{T}^{n}}|\xi|^{q(x)}dx \right) d\xi.
\end{split}
\end{align}
Now we concentrate on the first summation term in the above integral.
\begin{align*}
& \sum_{j=0}^{\infty}\sum_{k=1}^{\infty}\frac{(\alpha\delta^{-1})^{k}}{k!}\left(
\int_{\mathbb{T}^{n}} 2^{jq(x)\beta}|\xi|^{q(x)}dx \right)^{k}  \\
= & \sum_{k=1}^{\infty} \frac{(\alpha\delta^{-1})^{k}}{k!}\left( \sum_{j=0}^{\infty}\left( \int_{\mathbb{T}^{n}}
2^{jq(x)\beta}|\xi|^{q(x)}dx \right)^{k} \right)^{\frac{1}{k}k} \\
\leq &
\sum_{k=1}^{\infty}\frac{(\alpha\delta^{-1})^{k}}{k!}
\left( \int_{\mathbb{T}^{n}} \left( \sum_{j=0}^{\infty}2^{jq(x)\beta k} \right)^{1/k} |\xi|^{q(x)} dx \right)^{k}   \\
\leq &
\frac{1}{1-2^{\beta q^{-}}} \sum_{k=1}^{\infty}\frac{(\alpha\delta^{-1})^{k}}{k!}
\left( \int_{\mathbb{T}^{n}} |\xi|^{q(x)} dx \right)^{k}    \\
\leq &
\frac{1}{1-2^{\beta q^{-}}} \exp\left( \alpha\delta^{-1}\int_{\mathbb{T}^{n}} |\xi|^{q(x)} dx \right).
\end{align*}
Substituting the above inequality into (\ref{var_product_infinite}), we obtain that
\begin{align*}
& \sum_{j=0}^{\infty} \int_{\mathbb{R}} \left( \exp\left( \alpha\delta^{-1}\int_{\mathbb{T}^{n}} 2^{jq(x)\beta}|\xi|^{q(x)}dx \right) - 1 \right)
\exp\left( -\frac{1}{2}\int_{\mathbb{T}^{n}} |\xi|^{q(x)} dx \right) d\xi   \\
\leq & \frac{1}{1-2^{\beta q^{-}}} \int_{\mathbb{R}}\exp\left( \left(\alpha\delta^{-1} - \frac{1}{2}\right)
\int_{\mathbb{T}^{n}} |\xi|^{q(x)} dx \right)d\xi < \infty,
\end{align*}
where we used $\alpha\delta^{-1} < \frac{1}{2}$.

(2) $\Rightarrow$ (1).

If (1) does not hold, $\rho_{B_{q(\cdot)}^{t(\cdot)}}(u)$ is positive infinite on a set of positive measure $S$.
Then, since for $\alpha > 0$, $\exp\left( \alpha \rho_{B_{q(\cdot)}^{t(\cdot)}}(u) \right) = + \infty$ if
$\rho_{B_{q(\cdot)}^{t(\cdot)}}(u) = + \infty$, and
$$\mathbb{E}\left( \exp\left( \alpha \rho_{B_{q(\cdot)}^{t(\cdot)}}(u) \right) \right) \geq
\mathbb{E}\left( 1_{S}\exp\left( \alpha \rho_{B_{q(\cdot)}^{t(\cdot)}}(u) \right) \right),$$
we get a contradiction.

(1) $\Rightarrow$ (3).

Since $\rho_{B_{q(\cdot)}^{t(\cdot)}}(u) < \infty$, we easily know
\begin{align*}
\int_{\mathbb{T}^{n}}\sum_{j=0}^{\infty}\sum_{G\in G^{j}}\sum_{m\in \mathbb{M}_{j}}
2^{jq(x)(t(2^{-j}m)-s(2^{-j}m)+n/q^{+})} |\xi_{Gm}^{j}|^{q(x)} \chi_{jm}(x)dx < + \infty.
\end{align*}
Hence, for almost all $x \in \mathbb{T}^{n}$, the integrand in the above formula is finite.
Choose $x \in \mathbb{T}^{n}$ such that
\begin{align*}
\sum_{j=0}^{\infty}\sum_{G\in G^{j}}\sum_{m\in \mathbb{M}_{j}}
2^{jq(x)(t(2^{-j}m)-s(2^{-j}m)+n/q^{+})} |\xi_{Gm}^{j}|^{q(x)} \chi_{jm}(x) < \infty.
\end{align*}
So for every $j$, there is $m = m_{x,j}$ such that
\begin{align*}
\sum_{j=0}^{\infty}\sum_{G\in G^{j}}
2^{jq(x)(t(2^{-j}m_{x,j})-s(2^{-j}m_{x,j})+n/q^{+})} |\xi_{Gm_{x,j}}^{j}|^{q(x)} < \infty.
\end{align*}
If $t(2^{-j}m_{x,j})-s(2^{-j}m_{x,j})+n/q^{+} \geq 0 $, we have
\begin{align*}
\sum_{j=0}^{\infty}\sum_{G\in G^{j}} |\xi_{Gm_{x,j}}^{j}|^{q(x)} < \infty.
\end{align*}
Since there exist $c,C$ such that $0 < c \leq \mathbb{E}(|\xi_{Gm}^{j}|^{q(x)}) \leq C < \infty$ for every $\{j, G, m\}$,
this contradicts the law of large numbers e.g. Theorem \ref{low_of_large_number}.
So we obtain $t(2^{-j}m_{x,j})-s(2^{-j}m_{x,j})+n/q^{+} < 0 $ for infinite $j$.
By the definition of $\chi_{jm}(\cdot)$, we know $2^{-j}m_{x,j} \rightarrow x$, hence, we find that
\begin{align*}
t(x) - s(x) + \frac{n}{q^{+}} \leq 0.
\end{align*}
In addition, by our assumption and the continuity of $t(\cdot)$ and $s(\cdot)$, we finally obtain
\begin{align*}
t(x) - s(x) + \frac{n}{q^{+}} < 0
\end{align*}
for every $x \in \mathbb{T}^{n}$.
\end{proof}
\begin{remark}
Theorem \ref{convergence_series_if_only_if} assumed $t(x)-s(x)+\frac{n}{q^{+}} \neq 0$ and $\kappa(\cdot)$ is a uniform
probability distribution which seems technical, however, for the constant $q,t,s$ case these conditions are all satisfied
naturally. How to remove these conditions will be left to our future work.
\end{remark}

In the previous two theorems, we proved basic properties for random variables construct from infinite series (\ref{full_series}).
Now we study the situation where the family $\tilde{\Psi}_{Gm}^{j}$ have a uniform H\"{o}lder exponent $\alpha$
and study the implications for H\"{o}lder continuity of the random function $u$.
We assume that there are $C, a, b > 0$ and $\alpha \in (0,1]$ such that, for all $j \geq 0$,
\begin{align}\label{var_holder_basis}
\begin{split}
|\tilde{\Psi}_{Gm}^{j}(x)| & \leq C 2^{jn b}, \quad x\in \mathbb{T}^{n}, \\
|\tilde{\Psi}_{Gm}^{j}(x)-\tilde{\Psi}_{Gm}^{j}(y)| & \leq C 2^{jn a} |x-y|^{\alpha}, \quad x,y\in \mathbb{T}^{n}.
\end{split}
\end{align}
We also assume that $a > b$ as in \cite{Dashti_Stuart}.
\begin{theorem}\label{holder_continuity_theorem}
Assume that $u$ is given by (\ref{full_series}) and (\ref{coefficient_de}) with $\xi_{Gm}^{j}$ draw from a
centered generalized $q(\cdot)$-exponential distribution. Suppose also that (\ref{var_holder_basis}) hold and that
$s\in C_{loc}^{log}(\mathbb{T}^{n}) \cap L^{\infty}(\mathbb{T}^{n})$,
$q \in \mathcal{P}^{log}(\mathbb{T}^{n})$,
$\inf_{x\in \mathbb{T}^{n}} s(x) > n \left( b + \frac{1}{q^{+}} + \frac{1}{2}\theta (a-b) \right)$ for some
$\theta \in (0,2)$. Then $\mathbb{P}$-a.s. we have $u \in C^{\beta}(\mathbb{T}^{n})$ for all $\beta < \frac{\alpha\theta}{2}$.
\end{theorem}
\begin{proof}
Here we need to use Theorem \ref{kolmogorov_test} listed in the Appendix which is a variant of the
Kolmogorov continuity theorem. Denoted as in the Theorem \ref{kolmogorov_test} but use our series (\ref{full_series}), we obtain
\begin{align*}
S_{1} & = \sum_{j=0}^{\infty}\sum_{G\in G^{j}}\sum_{m\in \mathbb{M}_{j}} |\gamma_{Gm}^{j}|^{2} \|\tilde{\Psi}_{Gm}^{j}\|_{L^{\infty}}^{2}  \\
& \leq C \sum_{j=0}^{\infty}\sum_{G\in G^{j}}\sum_{m\in \mathbb{M}_{j}} 2^{-j\left(2s(2^{-jm})+n+\frac{2n}{q^{+}}\right)}2^{2jnb}   \\
& \leq C \sum_{j=0}^{\infty} 2^{jn}2^{-j\left(2s(2^{-jm})+n+\frac{2n}{q^{+}}\right)}2^{2jnb} = C \sum_{j=0}^{\infty}2^{-jnc_{1}}
\end{align*}
and
\begin{align*}
S_{2} & \leq C \sum_{j=0}^{\infty}\sum_{G\in G^{j}}\sum_{m\in \mathbb{M}_{j}} |\gamma_{Gm}^{j}|^{2-\theta}
\|\tilde{\Psi}_{Gm}^{j}\|_{L^{\infty}}^{2-\theta} |\gamma_{Gm}^{j}|^{\theta} 2^{jna\theta}  \\
& \leq C \sum_{j=0}^{\infty}2^{-j\left(s(2^{-jm})-\frac{n}{q^{+}}\right)}2^{j2nb}2^{-j\theta (b-a)n}
= C \sum_{j=0}^{\infty}2^{-jnc_{2}},
\end{align*}
where
\begin{align*}
c_{1} = \frac{2s(2^{-jm})}{n} - \frac{2}{q^{+}} - 2b > 0,
\end{align*}
and
\begin{align*}
c_{2} = \frac{2s(2^{-jm})}{n} - \frac{2}{q^{+}} - 2b - \theta (a-b) > 0.
\end{align*}
We need $c_{1}>0$ and $c_{2}>0$ and by our assumption $a > b$, we only need $c_{2} > 0$.
Our assumption $\inf_{x\in \mathbb{T}^{n}} s(x) > n \left( b + \frac{1}{q^{+}} + \frac{1}{2}\theta (a-b) \right)$
just assure $c_{2} > 0$. So by Theorem \ref{kolmogorov_test}, we can conclude our proof.
\end{proof}
\begin{remark}
If let the mean function is nonzero and satisfies
\begin{align*}
|m_{0}(x)| & \leq C, \quad x\in D,    \\
|m_{0}(x)-m_{0}(y)| & \leq C |x-y|^{\alpha}, \quad x,y \in D.
\end{align*}
Then the result of Theorem \ref{holder_continuity_theorem} still holds.
\end{remark}

\begin{theorem}\label{besov_continuity}
Assume that $u$ is given by (\ref{full_series}) and (\ref{coefficient_de}) with $\xi_{Gm}^{j}$ drawn from a centered generalized
$q(\cdot)$-exponential distribution. Suppose also that $\tilde{\Psi}_{Gm}^{j}$, with $\{j=0,1,\cdots,\infty, G \in G^{j}, m\in \mathbb{M}_{j}\}$ form
a basis for $B^{t(\cdot)}_{q(\cdot)}$ with $t^{-} > 0$, $q(\cdot) \in \mathcal{P}^{log}(\mathbb{T}^{n})$ and
$t \in C_{loc}^{log}(\mathbb{T}^{n}) \cap L^{\infty}(\mathbb{T}^{n})$.
Then for any $$\sup_{x\in \mathbb{T}^{n}} \left( t(x) - s(x) + \frac{n}{q^{+}} \right) < 0,$$
we have $u\in C^{t(\cdot)}(\mathbb{T}^{n})$ $\mathbb{P}$-almost surely.
\end{theorem}
\begin{proof}
For any $k \geq 1$, using the definition of $\rho_{B_{q(\cdot)}^{t(\cdot)}}(u)$, we can write
\begin{align*}
\rho_{B_{kq(\cdot)}^{t(\cdot)}}(u)
& = C_{\delta,m}\int_{\mathbb{T}^{n}}\sum_{j=0}^{\infty}\sum_{G\in G^{j}}\sum_{m\in \mathbb{M}_{j}}
2^{jkq(x)\left( t(2^{-j}m) - s(2^{-j}m) + \frac{n}{q^{+}} \right)} |\xi_{Gm}^{j}|^{kq(x)} \chi_{jm}(x) dx.
\end{align*}
For every $k \in \mathbb{N}$ there exists constants $c_{m}, C_{m}$ such that
$0<c_{m}\leq \mathbb{E}\left( |\xi_{Gm}^{j}|^{q(x)} \right)\leq C_{m} < \infty$.
Since each term of the above series is measurable we can swap the sum and the integration and obtain
\begin{align*}
\mathbb{E}\left(\rho_{B_{kq(\cdot)}^{t(\cdot)}}(u)\right) \leq C_{\delta,m}
\sum_{j=0}^{\infty} 2^{jq(x)q^{-}\left( t(2^{-j}m) - s(2^{-j}m) + \frac{n}{q^{+}} \right)} < \infty.
\end{align*}
From the above inequality, we obtain that $\rho_{B_{kq(\cdot)}^{t(\cdot)}}(u) < \infty$ $\mathbb{P}$-almost surely.
So we know that $\|u\|_{B_{kq(\cdot)}^{t(\cdot)}(\mathbb{T}^{n})} < \infty$ $\mathbb{P}$-almost surely.
Since
$$\sup_{x\in \mathbb{T}^{n}} \left( t(x) - s(x) + \frac{n}{q^{+}} \right) < 0,$$
we can choose $k$ large enough so that $\frac{n}{kq(x)} < s(x) - \frac{n}{q(x)} - t(x)$.
Then  the embedding $B_{kq(\cdot)}^{t_{1}(\cdot)} \hookrightarrow C^{t(\cdot)}(\mathbb{T}^{n})$ \cite{Almeida_Hasto}
for any $t_{1}$ satisfying $t(x) + \frac{n}{kq(x)} < t_{1}(x) < s(x) - \frac{n}{q(x)}$ implies that
$\|u\|_{C^{t(\cdot)}(\mathbb{T}^{n})} < \infty$ $\mathbb{P}$-almost surely.
It follows that $u \in C^{t(\cdot)}(\mathbb{T}^{n})$ $\mathbb{P}$-almost surely.
\end{proof}
\begin{remark}
If the mean function $m_{0}$ belongs to $C^{t(\cdot)}(\mathbb{T}^{n})$, the result of the above theorem holds
for a random series with nonzero mean function as well.
\end{remark}


\section{Bayesian approach to inverse problems}\label{bayesion_section}

In this section, we generalize results in \cite{inverse_fluid_equation} and \cite{Besov_prior} to our setting.
Let $X$, $Y$ be separable Banach spaces, equipped with the Borel $\sigma$-algebra, and $\mathcal{G}: X\rightarrow Y$ a
measurable mapping. We wish to solve the inverse problem of finding $u$ from $y$ where
\begin{align}\label{bayesian_basic_equation}
y = \mathcal{G}(u) + \eta
\end{align}
and $\eta \in Y$ denotes noise.
Employing Bayesian approach to this problem, we let $(u,y) \in X \times Y$ be a random variable and compute $u|y$.
We specify the random variable $(u,y)$ as follows:
\begin{itemize}
  \item \textbf{Prior}: $u \sim \mu_{0}$ measure on $X$.
  \item \textbf{Noise}: $\eta \sim \mathbb{Q}_{0}$ measure on $Y$, and $\eta \perp u$.
\end{itemize}

The random variable $y|u$ is then distributed according to the measure $\mathbb{Q}_{u}$, the translate of $\mathbb{Q}_{0}$
by $\mathcal{G}(u)$. We assume throughout the following that $\mathbb{Q}_{u} \ll \mathbb{Q}_{0}$ for $u$ $\mu_{0}$-a.s.
Thus we may define some \textbf{potential} $\Phi : X\times Y \rightarrow \mathbb{R}$ so that
\begin{align}\label{bayesian_potential}
\frac{d\mathbb{Q}_{u}}{d\mathbb{Q}_{0}}(u) = \exp\left( -\Phi(u;y) \right),
\end{align}
and
\begin{align}\label{bayesian_uni}
\int_{Y} \exp\left( -\Phi(u;y) \right) \mathbb{Q}_{0}(dy) = 1.
\end{align}

In the previous section, we construct probability measure $\mu_{0}$ which are supported on a given variable order
Besov space $B_{q(\cdot)}^{t(\cdot)}$. Following, we show how to use of such priors $\mu_{0}$ may be combined
with properties of $\Phi$, defined above, to deduce the existence of a well-posed Bayesian inverse problem.
We firstly list the following conditions of $\Phi$:

\textbf{Assumptions 1}: Let $X$ and $Y$ be Banach spaces.
The function $\Phi : X\times Y \rightarrow \mathbb{R}$ satisfies:

(i) there is an $\alpha_{1} > 0$ and for every $r > 0$, an $M \in \mathbb{R}$, such that for all $u \in X$, and
for all $y \in Y$ such that $\|y\|_{Y} < r$,
\begin{align*}
\Phi(u;y) \geq M - \alpha_{1}\|u\|_{X};
\end{align*}

(ii) for every $r > 0$ there exists $K = K(r) > 0$ such that for all $u \in X$, $y\in Y$ with $\max\{ \|u\|_{X}, \|y\|_{Y} \} < r$
\begin{align*}
\Phi(u;y) \leq K;
\end{align*}

(iii) for every $r > 0$ there exists $L = L(r) > 0$ such that for all $u_{1},u_{2} \in X$ and $y \in Y$ with
$\max\{ \|u_{1}\|_{X}, \|u_{2}\|_{X}, \|y\|_{Y} \} < r$
\begin{align*}
|\Phi(u_{1};y)-\Phi(u_{2};y)| \leq L \|u_{1}-u_{2}\|_{X};
\end{align*}

(iv) there is an $\alpha_{2} > 0$ and for every $r > 0$ a $C \in \mathbb{R}$ such that for all $y_{1},y_{2} \in Y$ with
$\max\{ \|y_{1}\|_{Y},\|y_{2}\|_{Y} \} < r$ and for every $u \in X$
\begin{align*}
|\Phi(u,y_{1}) - \Phi(u,y_{2})| \leq \exp\left( \alpha_{2}\|u\|_{X} + C \right) \|y_{1} - y_{2}\|_{Y}.
\end{align*}

Now we can give the following theorem for well-defined problem.
\begin{theorem}\label{bayesian_well_defined}
Let $\Phi$ satisfy (\ref{bayesian_uni}) and Assumptions 1 (i)-(iii). Suppose that for some
$q \in \mathcal{P}^{log}(\mathbb{T}^{n})$, $t \in C_{loc}^{log}(\mathbb{T}^{n})\cap L^{\infty}(\mathbb{T}^{n})$, $B_{q(\cdot)}^{t(\cdot)}$ is
continuously embedded in $X$. There exists $\delta^{*} > 0$ such that if $\mu_{0}$ is a $(\delta, B^{s(\cdot)}_{q(\cdot)})$
measure with $$\sup_{x \in \mathbb{T}^{n}} \left( t(x)-s(x)+\frac{n}{q^{+}} \right) < 0$$ and $\delta > \delta^{*}$,
then $\mu^{y}$ is absolutely continuous with respect to $\mu_{0}$ and satisfies
\begin{align}\label{bayesian_absolute}
\frac{d\mu^{y}}{d\mu_{0}}(u) = \frac{1}{Z(y)} \exp\left( -\Phi(u;y) \right)
\end{align}
with the normalizing factor
\begin{align}\label{bayesian_normalize}
Z(y) = \int_{X} \exp\left( -\Phi(u;y) \right) \mu_{0}(du) < \infty.
\end{align}
The constant $\delta^{*} = 2 \max\{c_{e}^{q^{-}}, c_{e}^{q^{+}}\} \alpha_{1}$, where $c_{e}$ is the embedding constant satisfying
$\|u\|_{X} \leq c_{e} \|u\|_{B_{q(\cdot)}^{t(\cdot)}}$.
\end{theorem}
\begin{proof}
Define $\pi_{0}(du,dy) = \mu_{0}(du)\otimes \mathbb{Q}_{0}(dy)$ and $\pi(du,dy) = \mu_{0}(du)\mathbb{Q}_{u}(dy)$.
Assumption 1 (iii) gives continuity of $\Phi$ on $X$ and since $\mu_{0}(X) = 1$ we have that $\Phi : X \rightarrow \mathbb{R}$
is $\mu_{0}$-measurable. Therefore $\pi \ll \pi_{0}$ and $\pi$ has Rando-Nikodym derivative given by (\ref{bayesian_potential}).
This then by Theorem 6.29 of \cite{acta_numerica}, implies that $\mu^{y}(du)$ is absolutely continuous with respect to $\mu_{0}(du)$.
This same lemma also gives (\ref{bayesian_absolute}) provided that the normalization constant (\ref{bayesian_normalize}) is positive, which
we now establish. Since $\mu_{0}(B^{t(\cdot)}_{q(\cdot)}) = 1$, we note that all the integrals over $X$ may be replaced by integrals over
$B^{t(\cdot)}_{q(\cdot)}$ for any $\sup_{x\in \mathbb{T}^{n}} \left( t(x)-s(x)+\frac{n}{q^{+}} \right) < 0$.
By Assumption 1 (i), we note that there is $M = M(y)$ such that
\begin{align*}
Z(y) & = \int_{B^{t(\cdot)}_{q(\cdot)}} \exp\left( -\Phi(u;y) \right) \mu_{0}(du) \\
& \leq \int_{B^{t(\cdot)}_{q(\cdot)}} \exp\left( \alpha_{1}\|u\|_{X} - M \right) \mu_{0}(du)
\end{align*}
By Lemma 3.2.5 of \cite{variable_book}, we know that
\begin{align*}
\|u\|_{X}^{q^{-}} \leq c_{e}^{q^{-}} \rho_{B_{q(\cdot)}^{t(\cdot)}}(u)
\end{align*}
when $\rho_{B_{q(\cdot)}^{t(\cdot)}}(u) \geq 1$ or
\begin{align*}
\|u\|_{X}^{q^{+}} \leq c_{e}^{q^{+}} \rho_{B_{q(\cdot)}^{t(\cdot)}}(u)
\end{align*}
when $\rho_{B_{q(\cdot)}^{t(\cdot)}}(u) < 1$.
Hence, if $\|u\|_{X} \geq 1$ then we have
\begin{align*}
Z(y) & = \int_{B^{t(\cdot)}_{q(\cdot)}} \exp\left( -\Phi(u;y) \right) \mu_{0}(du) \\
& \leq \int_{B^{t(\cdot)}_{q(\cdot)}} \exp\left( \alpha_{1}\max\{c_{e}^{q^{-}}, c_{e}^{q^{+}}\}
\rho_{B_{q(\cdot)}^{t(\cdot)}}(u) - M \right) \mu_{0}(du).
\end{align*}
This upper bound is finite by Theorem \ref{convergence_series_if_only_if} since $\delta > 2 \alpha_{1}\max\{c_{e}^{q^{-}}, c_{e}^{q^{+}}\}$.
For $\|u\|_{X} \leq 1$, we have
\begin{align*}
Z(y) & = \int_{B^{t(\cdot)}_{q(\cdot)}} \exp\left( -\Phi(u;y) \right) \mu_{0}(du) \\
& \leq \int_{B^{t(\cdot)}_{q(\cdot)}} \exp(\alpha_{1}-M)\mu_{0}(du) < \infty.
\end{align*}
Now, we prove that the normalisation constant does not vanish.
Let $R = \mathbb{E}(\rho_{B_{q(\cdot)}^{t(\cdot)}}(u))$, we know that $R \in (0, \infty)$.
As $\rho_{B_{q(\cdot)}^{t(\cdot)}}(u)$ is a nonegative random variable, we obtain that
$\mu_{0}(\rho_{B_{q(\cdot)}^{t(\cdot)}}(u) < R) > 0$. Taking $r = \max\{\|y\|_{Y},R\}$, Assumption 1 (ii) ensure that
\begin{align*}
Z(y) & = \int_{B^{t(\cdot)}_{q(\cdot)}} \exp(-\Phi(u;y)) \mu_{0}(du)  \\
& \geq \int_{\tilde{b}^{t(\cdot)}_{q(\cdot)} < R} \exp(-K) \mu_{0}(du)  \\
& = \exp(-K)\mu_{0}(\rho_{B_{q(\cdot)}^{t(\cdot)}}(u) < R)
\end{align*}
which is positive.
\end{proof}

We now show the well-posedness of the posterior measure $\mu^{y}$ with respect to the data $y$.
Recall that the Hellinger metric $d_{Hell}$ is defined by
\begin{align*}
d_{Hell}(\mu, \mu') = \sqrt{\frac{1}{2} \int \left( \sqrt{\frac{d\mu}{d\nu}} - \sqrt{\frac{d\mu'}{d\nu}} \right)^{2} d\nu},
\end{align*}
where $\nu$ is reference measure and with respect to which both $\mu$ and $\mu'$ are absolutely continuous.
The Hellinger metric is independent of the choice of the reference measure $\nu$.
For a review of probability metrics we refer to \cite{probability_metric}.
The following theorem can be proved by using similar ideas from Theorem 3.3 of \cite{Besov_prior}.
The minor differences are that we need to use $\rho_{B_{q(\cdot)}^{t(\cdot)}}(u)$ instead of $B^{t(\cdot)}_{q(\cdot)}$
when we need to bound $\|u\|_{X}$. The same situation appeared in the proof of Theorem \ref{bayesian_well_defined},
so we here omit the details of the proof for concisely.

\begin{theorem}\label{bayesian_well_posedness}
Let $\Phi$ satisfy (\ref{bayesian_uni}) and Assumptions 1 (i)-(iv). Suppose that for some
$q \in \mathcal{P}^{log}(\mathbb{T}^{n})$, $t \in C_{loc}^{log}(\mathbb{T}^{n}) \cap L^{\infty}(\mathbb{T}^{n})$, $B_{q(\cdot)}^{t(\cdot)}$ is
continuously embedded in $X$. There exists $\delta^{*} > 0$ such that if $\mu_{0}$ is a $(\delta, B^{s(\cdot)}_{q(\cdot)})$
measure with $$\sup_{x \in \mathbb{T}^{n}} \left( t(x)-s(x)+\frac{n}{q^{+}} \right) < 0$$ and $\delta > \delta^{*}$, then
\begin{align*}
d_{Hell}(\mu^{y}, \mu^{y'}) \leq C \|y-y'\|_{Y}
\end{align*}
where $C = C(r)$ with $\max\{ \|y\|_{Y}, \|y'\|_{Y} \} \leq r$.
The constant $\delta^{*} = 2 \max\{c_{e}^{q^{-}}, c_{e}^{q^{+}}\} (\alpha_{1} + 2\alpha_{2})$,
where $c_{e}$ is the embedding constant satisfying $\|u\|_{X} \leq c_{e} \|u\|_{B^{t(\cdot)}_{q(\cdot)}}$.
\end{theorem}

In the last part of this section, we consider the approximation of the posterior.
Consider $\Phi^{N}$ to be an approximation of $\Phi$. Here we state a result which quantifies the effect of this approximation
in the posterior measure in terms of the approximation error in $\Phi$.

Define $\mu^{y,N}$ by
\begin{align}\label{bayesian_approximation}
\frac{d\mu^{y,N}}{d\mu_{0}}(u) = \frac{1}{Z^{N}(y)} \exp \left( -\Phi^{N}(u) \right),
\end{align}
and
\begin{align}\label{bayesian_appro_normal}
Z^{N}(y) = \int_{X} \exp\left( -\Phi^{N}(u) \right)\mu_{0}(du).
\end{align}
We suppress the dependence of $\Phi$ and $\Phi^{N}$ on $y$ here as it is considered fixed.

\begin{theorem}\label{bayesian_approximation_theorem}
Assume that the measures $\mu$ and $\mu^{N}$ are both absolutely continuous with respect to $\mu_{0}$, and given by
(\ref{bayesian_absolute}) and (\ref{bayesian_approximation}) respectively. Suppose that $\Phi$ and $\Phi^{N}$ satisfy
Assumption 1 (i) and (ii), uniformly in $N$, and that there exist $\alpha_{3} \geq 0$ and $C \in \mathbb{R}$ such that
\begin{align*}
|\Phi(u) - \Phi^{N}(u)| \leq \exp(\alpha_{3}\|u\|_{X} + C)\varphi(N),
\end{align*}
where $\varphi(N) \rightarrow 0$ as $N \rightarrow \infty$. Suppose that for some
$t \in C_{loc}^{log}(\mathbb{T}^{n})\cap L^{\infty}(\mathbb{T}^{n})$,
$q\in \mathcal{P}^{log}(\mathbb{T}^{n})$, $B^{t(\cdot)}_{q(\cdot)}$ is continuously embedded in $X$.
Let $\mu_{0}$ be a $(\delta, B^{s(\cdot)}_{q(\cdot)})$ measure with
$$\sup_{x\in \mathbb{T}^{n}}\left( t(x) - s(x) + \frac{n}{q^{+}} \right) < 0$$ and
$\delta > 2 \max\{c_{e}^{q^{-}}, c_{e}^{q^{+}}\}(\alpha_{1}+2\alpha_{3})$ where
$c_{e}$ is the embedding constant satisfying $\|u\|_{X} \leq c_{e} \|u\|_{B^{t(\cdot)}_{q(\cdot)}}$.
Then there exists a constant independent of $N$ such that
\begin{align*}
d_{Hell}(\mu, \mu^{N}) \leq C \varphi(N).
\end{align*}
\end{theorem}
The proof of the above theorem similar to the proof of Theorem 3.3 of \cite{acta_numerica}, the differences
and difficulties can be overcome by same idea used in the proof of Theorem \ref{bayesian_well_defined}.


\section{Variational methods}\label{variational_section}

MAP estimator in the Bayesian statistics literature \cite{book_comp_bayeisn} is an important concept,
it specifies the relationship between Bayesian approach and classical regularization technique.
It is well known that for non-gaussian prior we can hardly obtain a rigorous relation between prior measure and
regularization term in infinite dimensions. Even for a simpler constant index Besov prior there is no complete theory \cite{Besov_prior}.
Only recently, Tapio Helin and Martin Burger \cite{MAP_Besov} addressed this issue in some sense.
Here, as stated in Remark \ref{intuitive_density}, we can get an upper bound for the probability density, and for constant
$q$ case we can get an equality.
So we faced more complex situation compared with constant Besov case
and in this section we just give some partial illustration.

We define the following functional
\begin{align}\label{map_functional}
I(u) = \Phi(u) + \frac{1}{2}\rho_{B_{q(\cdot)}^{s(\cdot)}}(u).
\end{align}
Intuitively, the minimizers of the above functional or some variant of (\ref{map_functional}) may have
highest probability measure when a small balls centred on such minimizers.
For more explanations, we refer to the Gaussian case \cite{MAP_detail}.
For this functional we have
\begin{theorem}\label{convergence}
Let Assumption 1 (i),(ii) hold, assume that $\mu_{0}(X) = 1$.
Let $s\in C_{loc}^{log}(\mathbb{T}^{n})\cap L^{\infty}(\mathbb{T}^{n})$, $q\in \mathcal{P}^{log}(\mathbb{T}^{n})$, $1 < q^{-}\leq q(x) \leq q^{+} < \infty$
and $B^{s(\cdot)}_{q(\cdot)}$ compactly embedded in $X$.
Then there exists $\bar{u} \in B^{s(\cdot)}_{q(\cdot)}$ such that
\begin{align*}
I(\bar{u}) = \bar{I} := \inf\left\{I(u) : u \in B^{s(\cdot)}_{q(\cdot)}\right\}.
\end{align*}
Furthermore, if $\{u_{n}\}$ is a minimizing sequence satisfying $I(u_{n}) \rightarrow I(\bar{u})$
then there is a subsequence $\{u_{n'}\}$ that converges strongly to $\bar{u}$ in $B^{s(\cdot)}_{q(\cdot)}$.
\end{theorem}

Before proving this theorem, we need to prove the following lemma which is of independent interests.
\begin{lemma}\label{reflect}
Let $s\in C_{loc}^{log}(\mathbb{T}^{n})\cap L^{\infty}(\mathbb{T}^{n})$, $q\in \mathcal{P}^{log}(\mathbb{T}^{n})$, $1 < q^{-}\leq q(x) \leq q^{+} < \infty$,
then the dual space of $B^{s(\cdot)}_{q(\cdot)}$ is $B^{-s(\cdot)}_{q(\cdot)'}$ where
\begin{align*}
\frac{1}{q(x)} + \frac{1}{q(x)'} = 1.
\end{align*}
\end{lemma}
\begin{proof}
Step 1. Let $s\in C_{loc}^{log}\cap L^{\infty}$, $q \in \mathcal{P}^{log}$. We prove in this step that
$B^{-s(\cdot)}_{q(\cdot)}(\mathbb{R}^{n}) \subset B^{s(\cdot)}_{q(\cdot)}(\mathbb{R}^{n})'$
where $B^{s(\cdot)}_{q(\cdot)}(\mathbb{R}^{n})'$ stands for the dual space of
$B^{-s(\cdot)}_{q(\cdot)}(\mathbb{R}^{n})$.
Let $f \in B^{-s(\cdot)}_{q'(\cdot)}(\mathbb{R}^{n})$, define
$f_{k} = \sum_{r = -1}^{1} (\varphi_{k+r}\hat{f})^{\vee}$
where $\varphi_{\ell}$ defined as in (\ref{var_define_eql_space}) with $\varphi$
is a smooth decompositions of unity \cite{DanchinBook,Triebel2008}. Then we know that
\begin{align}
f = \sum_{k = 0}^{\infty} (\varphi_{k}\hat{f_{k}})^{\vee} \quad \text{in }\mathcal{S}'(\mathbb{R}^{n})
\end{align}
and
\begin{align}
\begin{split}
& \|2^{-s(\cdot)}f_{k}\|_{L^{q'(\cdot)}(\mathbb{R}^{n}, \ell^{q'(\cdot)})}  \\
= & \inf\left\{ \lambda > 0 : \, \int_{\mathbb{R}^{n}}
\sum_{j = 0}^{\infty} \frac{2^{-js(x)q'(x)}}{\lambda^{q'(x)}} |(\varphi_{j}\hat{f}_{k})^{\vee}(x)|^{q'(x)} dx \leq 1 \right\}   \\
\leq & C \inf\left\{ \frac{\lambda}{2} > 0 : \, \int_{\mathbb{R}^{n}}
\sum_{j = 0}^{\infty} \frac{2^{-js(x)q'(x)}}{\left(\frac{\lambda}{2}\right)^{q'(x)}} |(\varphi_{k}\hat{f})^{\vee}(x)|^{q'(x)} dx \leq 1 \right\}    \\
\leq & C \|f\|_{B^{-s(\cdot)}_{q(\cdot)}}.
\end{split}
\end{align}
Take $\varphi \in \mathcal{S}(\mathbb{R}^{n})$, then we have
\begin{align*}
|f(\varphi)| & = \left|\sum_{k=0}^{\infty}f(\mathcal{F}(\varphi_{k}\varphi^{\vee}))\right|  \\
& = \left| \sum_{k=0}^{\infty}\sum_{r=-1}^{1} (\varphi_{k+r}\hat{f}_{k+r})^{\vee}(\mathcal{F}(\varphi_{k}\varphi^{\vee})) \right|   \\
& = \left| \sum_{\ell = 0}^{\infty} \sum_{r=-1}^{1}\int_{\mathbb{R}^{n}} 2^{-s(x)\ell}f_{\ell}(x)2^{s(x)\ell}
\mathcal{F}(\varphi_{\ell}\varphi_{\ell+r}\varphi^{\vee}) dx \right|    \\
& \leq C \|2^{-s(\cdot)k}f_{k}\|_{L^{q'(\cdot)}(\mathbb{R}^{n},\ell^{q'(\cdot)})}
\|2^{s(\cdot)k}\mathcal{F}(\varphi_{k}\varphi_{k+r}\varphi^{\vee})\|_{L^{q(\cdot)}(\mathbb{R}^{n},\ell^{q(\cdot)})} \\
& \leq C \|f\|_{B^{-s(\cdot)}_{q(\cdot)}(\mathbb{R}^{n})}\|\varphi\|_{B^{s(\cdot)}_{q(\cdot)}(\mathbb{R}^{n})}.
\end{align*}
Hence, now we proved $B^{-s(\cdot)}_{q(\cdot)}(\mathbb{R}^{n}) \subset B^{s(\cdot)}_{q(\cdot)}(\mathbb{R}^{n})'$.

Step 2. In this step, we need to prove $B^{s(\cdot)}_{q(\cdot)}(\mathbb{R}^{n})' \subset B^{-s(\cdot)}_{q(\cdot)}(\mathbb{R}^{n})$.
Since $f \in B^{s(\cdot)}_{q(\cdot)}(\mathbb{R}^{n}) \rightarrow \{2^{s(\cdot)k}(\varphi_{k}\hat{f})^{\vee}\}_{k=0}^{\infty}$
is a one-to-one mapping from $B^{s(\cdot)}_{q(\cdot)}(\mathbb{R}^{n})$ onto a subspace of $L^{q(\cdot)}(\mathbb{R}^{n},\ell^{q(\cdot)})$,
every functional $g \in (B^{s(\cdot)}_{q(\cdot)})'$ can be interpreted as a functional on that subspace.
By the Hahn-Banach theorem, $g$ can be extended to a continuous linear functional on $L^{q(\cdot)}(\mathbb{R}^{n},\ell^{q(\cdot)})$,
where the norm of $g$ is preserved. If $\varphi \in \mathcal{S}(\mathbb{R}^{n})$,
considering Corollary 1 of Theorem 8 in Chapter 13 of \cite{vector_measure}
we have
\begin{align}\label{dual_two}
g(\varphi) = \int_{\mathbb{R}^{n}} \sum_{k=0}^{\infty}g_{k}(x)(\varphi_{k}\hat{\varphi})^{\vee}(x)dx
\end{align}
where $\|2^{-s(\cdot)k}g_{k}\|_{L^{q'(\cdot)}(\mathbb{R}^{n},\ell^{q'(\cdot)})}$ is equivalent
to the operator norm of $g$.
(\ref{dual_two}) implies that
\begin{align}\label{newform_dual_two}
g = \sum_{k=0}^{\infty} \mathcal{F}(\varphi_{k}g_{k}^{\vee}) \quad \text{in }\mathcal{S}'.
\end{align}
Define $\eta_{km}(x) = \frac{2^{nk}}{(1+2^{k}|x|)^{m}}$ with $m$ is a large enough constant,
since $\varphi$ can be chosen to be a radial smooth function with
compact support, we know that it can be controlled as follows
\begin{align}\label{dual_control}
2^{kn}\varphi^{\vee}(2^{k}x) \leq C \eta_{km}(x).
\end{align}
Now, we have
\begin{align*}
& \|2^{-s(\cdot)k}(\varphi_{k}\hat{g})^{\vee}\|_{L^{q'(\cdot)}(\mathbb{R}^{n},\ell^{q'(\cdot)})}  \\
= & \inf\left\{ \lambda > 0:\, \int_{\mathbb{R}^{n}} \frac{1}{\lambda^{q'(x)}}
\sum_{k=0}^{\infty} 2^{-s(x)kq'(x)} |(\varphi_{k}\hat{g})^{\vee}(x)|^{q'(x)} dx \leq 1 \right\} \\
= & \inf\left\{ \lambda > 0:\, \int_{\mathbb{R}^{n}} \frac{1}{\lambda^{q'(x)}}
\sum_{k=0}^{\infty} 2^{-s(x)kq'(x)} |\sum_{r=-1}^{1}(\varphi_{k}\varphi_{k+r}g_{k+r}^{\vee}(-\cdot))^{\vee}(x)|^{q'(x)} dx \leq 1 \right\} \\
\leq & C \|2^{-s(\cdot)k} \varphi_{k}^{\vee} * g_{k}\|_{L^{q'(\cdot)}(\mathbb{R}^{n},\ell^{q'(\cdot)})}    \\
\leq & C \|2^{-s(\cdot)k} \eta_{km} * g_{k}\|_{L^{q'(\cdot)}(\mathbb{R}^{n},\ell^{q'(\cdot)})}     \\
\leq & C \|2^{-s(\cdot)k} g_{k}\|_{L^{q'(\cdot)}(\mathbb{R}^{n},\ell^{q'(\cdot)})}
\end{align*}
where the last inequality follows from Lemma 5.4 of \cite{variable_tribel}.
(Here we need a small modification of Lemma 5.4 of \cite{variable_tribel}, however, the modification
is straightforward so we omit for concise.)
With the above estimates, obviously, the proof is completed.
\end{proof}
Although the above proof is for the whole space $\mathbb{R}^{n}$, it is also valid for periodic domain $\mathbb{T}^{n}$.
Now, let us come back to the proof of Theorem \ref{convergence}.
\begin{proof}
For any $\delta > 0$ there is $N = N_{1}(\delta)$ such that
\begin{align*}
\bar{I} \leq I(u_{n}) \leq \bar{I} + \delta, \quad \forall n\geq N_{1}.
\end{align*}
Thus
\begin{align*}
\frac{1}{2}\rho_{B_{q(\cdot)}^{s(\cdot)}}(u_{n}) \leq \bar{u} + \delta \quad \forall n \geq N_{1}.
\end{align*}
The sequence $\{u_{n}\}$ is bounded in $B^{s(\cdot)}_{q(\cdot)}$.
By the above Lemma \ref{reflect}, we know that $B^{s(\cdot)}_{q(\cdot)}$ is reflexive, there exists
$\bar{u} \in B^{s(\cdot)}_{q(\cdot)}$ such that $u_{n} \rightharpoonup \bar{u}$ in $B^{s(\cdot)}_{q(\cdot)}$.
By the compact embedding of $B^{s(\cdot)}_{q(\cdot)}$ in $X$ we deduce that $u_{n} \rightarrow \bar{u}$, strongly in $X$.
Noticing that $\rho_{B_{q(\cdot)}^{s(\cdot)}}(u)$ is lower semi-continuous by Theorem 2.2.8 of \cite{variable_book}.
Now, we can use similar ideas of Theorem 2.7 in \cite{inverse_fluid_equation} to complete the proof.
\end{proof}


\section{Application to backward diffusion problem}\label{application}

In this section, we apply our theory to inverse source problems for integer order diffusion equation and
fractional order diffusion equation.

\subsection{Integer order diffusion equation}

For simplicity we consider periodic domain $\mathbb{T}^{n}$. Define the operator $A$ as follows:
\begin{align*}
& H = \left( L^{2}(\mathbb{T}^{n}), <\cdot, \cdot>, \|\cdot\| \right) \\
& A = -\Delta, \quad \mathcal{D}(A) = H^{2}(\mathbb{T}^{n}).
\end{align*}
Consider the diffusion equation on $\mathbb{T}^{n}$ with periodic boundary condition, writing it as an ordinary
differential equation in $H$:
\begin{align}
\frac{d}{dt}v + A v = 0,\quad v(0) = u.
\end{align}
Define $G(u) = e^{-A}u$, $\ell$ to be an operator defined as follows
\begin{align}
\ell(G(u)) = \left( G(u)(x_{1}), G(u)(x_{2}), \cdots, G(u)(x_{K}) \right)^{T}
\end{align}
where $K$ is a fixed constant.
Then we have the relationship
\begin{align}\label{inverse_heat_form}
y = \ell (G(u)) + \eta
\end{align}
where $\eta = \{\eta_{j}\}_{j=1}^{K}$ is a mean zero Gaussian with covariance $\Gamma$, $y = \{y_{j}\}_{j=1}^{K}$
is the data which we measured.
Now we can show the well-definedness of the posterior measure and its continuity with respect to the data for
the above inverse diffusion problem.
\begin{theorem}\label{inverse_heat_wellposedness}
Consider the inverse problem for finding $u$ from noisy observations of $G(u) = v(1,\cdot)$ in the form of (\ref{inverse_heat_form}).
Let $\mu_{0}$ to be distributed as a variable index Besov prior $(\delta, B^{s(\cdot)}_{q(\cdot)})$ with
$s\in C_{loc}^{log}(\mathbb{T}^{n})\cap L^{\infty}(\mathbb{T}^{n})$,
$q\in \mathcal{P}^{log}(\mathbb{T}^{n})$, $\inf_{x\in \mathbb{T}^{n}}s(x) > \frac{n}{q^{+}}$ and $\delta > 4$.
Then the measure $\mu^{y}(du)$ is absolutely continuous with respect to $\mu_{0}$ with Radon-Nikodym derivative satisfying
\begin{align*}
\frac{d\mu^{y}}{d\mu_{0}}(u) = \frac{1}{Z(y)} \exp(-\Phi(u;y))
\end{align*}
where
\begin{align*}
\Phi(u;y) = \frac{1}{2}|\Gamma^{-1/2}(y-\ell(G(u)))|^{2} - \frac{1}{2}|\Gamma^{-1/2}y|^{2}
\end{align*}
and
\begin{align*}
Z(y) = \int_{B^{t(\cdot)}_{q(\cdot)}}\exp\left(-\frac{1}{2}|\Gamma^{-1/2}(y-\ell(G(u)))|^{2}
+ \frac{1}{2}|\Gamma^{-1/2}y|^{2}\right)\mu_{0}(du)
\end{align*}
with $\sup_{x\in \mathbb{T}^{n}}\left(t(x) -s(x) + \frac{n}{q^{+}} \right)<0$.
Furthermore, the posterior measure is continuous in the Hellinger metric with respect to the data
\begin{align*}
d_{Hell}(\mu^{y}, \mu^{y'}) \leq C |y-y'|.
\end{align*}
\end{theorem}
\begin{proof}
Firstly, we prove two properties about the operator $\ell (G(\cdot))$.

\textbf{Property 1}:
For large enough $\ell > 0$ and a small constant $\epsilon > 0$, by Sobolev embedding theorem, we have
\begin{align}\label{inverse_heat_bounded}
\begin{split}
|\ell(G(u))| & \leq K \|G(u)\|_{L^{\infty}} = \|e^{-A}u\|_{L^{\infty}}    \\
& \leq \|A^{\ell}e^{-A}u\|_{L^{2}} + \|e^{-A}u\|_{L^{2}}   \\
& \leq C \|u\|_{B_{2,2}^{t^{-}-\frac{n}{q^{-}}+\frac{n}{2}-\epsilon}}   \\
& \leq C \|u\|_{B_{q(\cdot)}^{t(\cdot)}}
\end{split}
\end{align}
where we used the fact that $A^{\gamma}e^{-\lambda A}$, $\lambda > 0$, is a bounded linear operator
from $B_{2,2}^{a}$ to $B_{2,2}^{b}$, any $a,b,\gamma \in \mathbb{R}$ and we also used embedding theorems
for variable index Besov space \cite{Almeida_Hasto}.

\textbf{Property 2}:
Let $u_{1}$, $u_{2}$ be two different initial data for heat equations.
By similar idea from the proof of Property 1, we have
\begin{align}\label{inverse_heat_lipschitz}
\begin{split}
|\ell(G(u_{1})) - \ell(G(u_{2}))| & \leq K \|G(u_{1}) - G(u_{2})\|_{L^{\infty}}   \\
& = K \|G(u_{1}-u_{2})\|_{L^{\infty}}   \\
& \leq C \|u_{1} - u_{2}\|_{B_{q(\cdot)}^{t(\cdot)}}.
\end{split}
\end{align}
Now let $X = B_{q(\cdot)}^{t(\cdot)}$. With the above Property 1 and Property 2, it is straightforward to know that
$\Phi(u;y)$ satisfy Assumption 1 (i)-(iv) with $\alpha_{1} = 0$ and $\alpha_{2} = 1$.
By Theorem \ref{bayesian_well_defined} and Theorem \ref{bayesian_well_posedness}, we immediately
obtain our desired results.
\end{proof}

\subsection{Fractional order diffusion equation}
For fractional order diffusion equations, there are numerous literature. For the well-posedness
theory, we refer to \cite{jia1,jia2,jia3,jia4}.
Here we just treat fractional diffusion equations on periodic domain as follows
\begin{align}\label{fractional_diffusion_equation}
\begin{split}
& \partial_{t}^{\alpha}v(t,x) + (-\Delta)^{\beta}v(t,x) = 0,\quad t\geq 0,\, x\in \mathbb{T}^{n}, \\
& v(0) = u,
\end{split}
\end{align}
where $0 < \alpha \leq 1$ and $0 < \beta \leq 1$ and $\partial_{t}^{\alpha}$ stands for Caputo derivarive of $\alpha$ order,
it can be defined as follows
\begin{align*}
\partial_{t}^{\alpha} f(t) := \frac{1}{\Gamma(1-\alpha)} \int_{0}^{t} (t-s)^{-\alpha}f'(s)ds,
\end{align*}
where $\Gamma(\cdot)$ is the usual Gamma function.
Define the operator $A$ as follows:
\begin{align*}
& H = \left( L^{2}(\mathbb{T}^{n}), <\cdot, \cdot>, \|\cdot\| \right) \\
& A = (-\Delta)^{\beta}, \quad \mathcal{D}(A) = H^{2\beta}(\mathbb{T}^{n}).
\end{align*}
Consider the heat conduction equation on $\mathbb{T}^{n}$ with periodic boundary condition, writing it as an ordinary
differential equation in $H$:
\begin{align}\label{fractional_time}
\partial_{t}^{\alpha}v + A v = 0,\quad v(0) = u.
\end{align}
If $A$ is a bounded operator e.g. a positive number, then the solution of the above equation (\ref{fractional_time}) has the following form
\begin{align*}
v(t) = E_{\alpha}(-At^{\alpha})u,
\end{align*}
where $E_{\alpha}(\cdot)$ is Mittag-Leffer function defined as
\begin{align*}
E_{\alpha}(z) = \sum_{k=0}^{\infty}\frac{z^{k}}{\Gamma(\alpha k+1)}.
\end{align*}
For more properties about Mittag-Leffer function we refer to \cite{jia5,jia6}.
In \cite{jia7}, they proposed fractional operator semigroup which characterize the solution of
abstract fractional cauchy problem (\ref{fractional_time}).
Since operator $A$ in our setting can generate a fractional operator semigroup, we can
define $G(u) = E_{\alpha}(-A)u$, $\ell$ to be an operator defined as follows
\begin{align}
\ell(G(u)) = \left( G(u)(x_{1}), G(u)(x_{2}), \cdots, G(u)(x_{K}) \right)^{T}
\end{align}
where $K$ is a fixed constant.
Then we have the relationship
\begin{align}\label{inverse_fractional_heat_form}
y = \ell (G(u)) + \eta
\end{align}
where $\eta = \{\eta_{j}\}_{j=1}^{K}$ is a mean zero Gaussian with covariance $\Gamma$, $y = \{y_{j}\}_{j=1}^{K}$
is the data which we measured.
Reviewing the proof of Theorem \ref{inverse_heat_wellposedness}, the key points are the estimates about the operator $\ell(G(\cdot))$.
Here, we meet more crucial situations for the fractional diffusion equations have no strong smoothing 
effect as normal diffusion equations. For a more complete illustration, we refer to \cite{fractional_inv}.
Mittag-Leffer function has only polynomial decay rate which restrict the smoothing effect. More precisely, 
we list the following decay rate estimates
\begin{lemma}\cite{jia6}\label{mit_decay}
If $0 < \alpha < 2$, $\mu$ is an arbitrary real number such that
\begin{align*}
\frac{\pi \alpha}{2} < \mu < \min \{ \pi, \pi\alpha \},
\end{align*}
then for an arbitrary integer $p \geq 1$,
when $|z| \rightarrow \infty$ the following expansion holds:
\begin{align*}
E_{\alpha}(z) = \frac{1}{\alpha}e^{z^{1/\alpha}} - \sum_{k=1}^{p}\frac{z^{-k}}{\Gamma(\beta - \alpha k)} + O(|z|^{-1-p}),
\end{align*}
where $|\text{arg}(z)| \leq \mu$.
\end{lemma}

Based on the above observation, we must restrict the fractional order $\alpha$, $\beta$ to some 
appropriate interval to gain enough smoothing effect to obtain the forward operator is lipschitz continuous. 
More precisely, we can obtain the following result.
\begin{theorem}\label{inverse_fractional_heat_wellposedness}
Consider the inverse problem for finding $u$ from noisy observations of $G(u) = v(1,\cdot)$ in the form of (\ref{inverse_fractional_heat_form})
with $0 < \alpha \leq 1$ and $\frac{n}{4} < \beta \leq 1$.
Let $X = L^{2}(\mathbb{T}^{n})$, $\mu_{0}$ to be distributed as a variable index Besov prior $(\delta, B^{s(\cdot)}_{q(\cdot)})$ with
$s\in C_{loc}^{log}(\mathbb{T}^{n}) \cap L^{\infty}(\mathbb{T}^{n})$,
$q\in \mathcal{P}^{log}(\mathbb{T}^{n})$, $\inf_{x\in \mathbb{T}^{n}}s(x) > \frac{n}{q^{+}}$ and $\delta > 4$.
Assume $t\in C_{loc}^{log}(\mathbb{T}^{n}) \cap L^{\infty}(\mathbb{T}^{n})$ and
\begin{align*}
\frac{n}{q^{-}} - \frac{n}{2} < t^{-} \leq t^{+} < s^{-} - \frac{n}{q^{+}}.
\end{align*} 
Then the measure $\mu^{y}(du)$ is absolutely continuous with respect to $\mu_{0}$ with Radon-Nikodym derivative satisfying
\begin{align*}
\frac{d\mu^{y}}{d\mu_{0}}(u) = \frac{1}{Z(y)} \exp(-\Phi(u;y))
\end{align*}
where
\begin{align*}
\Phi(u;y) = \frac{1}{2}|\Gamma^{-1/2}(y-\ell(G(u)))|^{2} - \frac{1}{2}|\Gamma^{-1/2}y|^{2}
\end{align*}
and
\begin{align*}
Z(y) = \int_{X}\exp\left(-\frac{1}{2}|\Gamma^{-1/2}(y-\ell(G(u)))|^{2}
+ \frac{1}{2}|\Gamma^{-1/2}y|^{2}\right)\mu_{0}(du).
\end{align*}
Furthermore, the posterior measure is continuous in the Hellinger metric with respect to the data
\begin{align*}
d_{Hell}(\mu^{y}, \mu^{y'}) \leq C |y-y'|.
\end{align*}
\end{theorem}
\begin{proof}
In order to prove the above theorem, we first give the following estimates. Denote $f_{\ell} (\ell \in Z^{n})$ to be the fourier 
coefficient of function $f$, then we have
\begin{align}\label{smooth}
\begin{split}
\|A E_{\alpha}(-A)f\|_{L^{2}(\mathbb{T}^{n})}^{2} & = \sum_{\ell \in Z^{n}} \left( |\ell|^{2\beta}E_{\alpha}(-|\ell|^{2\beta}) \right)^{2}
|f_{\ell}|^{2}  \\
& \leq C \|f\|_{L^{2}(\mathbb{T}^{n})}^{2},
\end{split}
\end{align}
where we use Lemma \ref{mit_decay}.
Considering $\frac{n}{q^{-}} - \frac{n}{2} < t^{-} \leq t^{+} < s^{-} - \frac{n}{q^{+}}$, for an arbitrary small 
positive number $\epsilon > 0$, we know that 
\begin{align*}
B^{t(\cdot)}_{q(\cdot)} \hookrightarrow B_{2,2}^{t^{-}-\frac{n}{q^{-}}+\frac{n}{2}-\epsilon} \hookrightarrow X.
\end{align*}
Using (\ref{smooth}), we easily have
\begin{align}\label{frac_inverse_heat_bounded}
\begin{split}
|\ell(G(u))| & \leq K \|G(u)\|_{L^{\infty}} = \|E_{\alpha}(-A)u\|_{L^{\infty}}    \\
& \leq \|A E_{\alpha}(-A)u\|_{L^{2}} + \|E_{\alpha}(-A)u\|_{L^{2}}   \\
& \leq C \|u\|_{L^{2}},
\end{split}
\end{align}
where we used $\beta > \frac{n}{4}$ to obtain the first inequality.
Similarly, we can obtain
\begin{align*}
|\ell(G(u_{1})) - \ell(G(u_{2}))| \leq C \|u_{1} - u_{2}\|_{L^{2}}.
\end{align*}
At this stage, we can complete the proof as in the integer case easily.
\end{proof}


\section{Conclusion}
In this paper, we firstly use the wavelet representations for function space on periodic domain to construct 
a probability measure called $(\delta, B^{s(\cdot)}_{q(\cdot)})$ measure. It can roughly be seen as 
a counterpart of variable regularization terms. Using the new non-gaussian measure as our priori measure,
we establish ``well-posendess'' theory for inverse problem as did in \cite{Dashti_Stuart}.
Through our study, we give another choice for the priori measure except Gaussian and Besov priori measure
and in addition, it may give an another understanding of variable order space regularization terms.

Secondly, we use our theory to integer order backward diffusion problems and fractional order backward diffusion problems.
Especially, we prove that for time derivative $0< \alpha \leq 1$ and space derivative $\frac{n}{2} < \beta \leq 2$ ($n$ is space dimension)
, the fractional order backward diffusion problem is ``well-posedness'' under Bayesian inverse framework.
This study also reflects that the fractional order problems is not a straight generalization of integer order problems.
When we consider fractional order problems, we must notice that sometimes the fractional order equations have 
totally different regularization properties compare to integer order equations.


\section{Technical lemmas}\label{yin_li}

\subsection{Properties of $\rho^{E}_{B^{t(\cdot)}_{q(\cdot)}}$ appeared in Section \ref{prior_section}}

\begin{lemma}\label{semimodular}
Let $1 \leq q^{-} \leq q(\cdot) \leq q^{+} < \infty$ and $t(\cdot) \in C(\mathbb{T}^{n})$.
Then $\rho^{E}_{B^{t(\cdot)}_{q(\cdot)}}$ is a modular and continuous.
\end{lemma}
\begin{proof}
Properties (1) and (2) in Definition \ref{definition_semimodular} are obviously satisfied.
To prove (3), we suppose that
\begin{align*}
\rho^{E}_{B^{t\cdot}_{q(\cdot)}} (\lambda u) = 0
\end{align*}
for all $\lambda > 0$.
Clearly, for some $k_{0}$
\begin{align*}
\int_{\Omega}\int_{\mathbb{T}^{n}} \lambda^{q(x)} 2^{k_{0} t(x) q(x)} |u_{k_{0}}(x,\omega)|^{q(x)} dx \mathbb{P}(d\omega)
\leq
\rho^{E}_{B^{t\cdot}_{q(\cdot)}} (\lambda u) = 0
\end{align*}
Since $1 \leq q^{-} \leq q(\cdot) \leq q^{+} < \infty$ and $t(\cdot) \in C(\mathbb{T}^{n})$, we easily obtain that
$u_{k_{0}} = 0$. Hence, we obtain $u = 0$.
Let $\mu \rightarrow 1$, we need to prove
$\rho^{E}_{B^{t(\cdot)}_{q(\cdot)}}(\lambda u) \rightarrow \rho^{E}_{B^{t(\cdot)}_{q(\cdot)}}(u)$.
Fix $\epsilon > 0$, choose $N > 0$, let $\mu < 1$ and close to $1$ enough such that
\begin{align*}
\rho^{E}_{B^{t(\cdot)}_{q(\cdot)}}(u) & <
\int_{\Omega} \int_{\mathbb{T}^{n}} \sum_{k = 0}^{N} 2^{k t(x) q(x)} |u_{k}(x,\omega)|^{q(x)} dx \mathbb{P}(d\omega) + \epsilon \\
& < \int_{\Omega} \int_{\mathbb{T}^{n}} \sum_{k = 0}^{N} \lambda^{q(x)} 2^{k t(x) q(x)} |u_{k}(x,\omega)|^{q(x)} dx \mathbb{P}(d\omega) + 2\epsilon \\
& < \rho^{E}_{B^{t(\cdot)}_{q(\cdot)}}(\lambda u) + 2\epsilon.
\end{align*}
Hence, we find $\rho^{E}_{B^{t(\cdot)}_{q(\cdot)}}(\lambda u)$ is left continuous with respect to $\lambda$.
Similar method can give the right continuous.
\end{proof}

\begin{lemma}\label{tec_convex}
Let $q \in \mathcal{P}$, then $\rho^{E}_{B^{t(\cdot)}_{q(\cdot)}}$ is convex.
\end{lemma}
\begin{proof}
Let $\theta \in (0,1)$, then
\begin{align*}
\rho^{E}_{B^{t(\cdot)}_{q(\cdot)}}(\theta f + (1-\theta)g) = &
\int_{\Omega}\int_{\mathbb{T}^{n}} \sum_{k=0}^{\infty} 2^{kt(x)q(x)} |\theta f_{k}(x,\omega) + (1-\theta)g_{k}(x,\omega)|^{q(x)} dx\mathbb{P}(d\omega) \\
\leq & \rho^{E}_{B^{t(\cdot)}_{q(\cdot)}}(\theta f) + \rho^{E}_{B^{t(\cdot)}_{q(\cdot)}}((1-\theta)g)   \\
\leq & \theta \rho^{E}_{B^{t(\cdot)}_{q(\cdot)}}(f) + (1-\theta) \rho^{E}_{B^{t(\cdot)}_{q(\cdot)}}(g).
\end{align*}
\end{proof}
In our case the parameter $p(\cdot)$ in \cite{Almeida_Hasto} is equal to $q(\cdot)$, so we only need
$1 \leq q^{-}$ not $2 \leq q^{-}$ indicated by Theorem 3.6 in \cite{Almeida_Hasto}.

\subsection{Proof of Lemma \ref{wavelet_characterization_measure}}

Concerning the proof of Lemma \ref{wavelet_characterization_measure}, we give the following two important Lemmas.
\begin{lemma}\label{basic_measure}
Let $1\leq q^{-} \leq q(\cdot) \leq q^{+} < \infty$, $\delta > 0$. For any sequence $\{ g_{j} \}_{j = 0}^{\infty}$
of nonnegative measurable functions on $\mathbb{T}^{n}$ denote
\begin{align*}
G_{j}(x,\omega) = \sum_{k = 0}2^{-|k-j|\delta}g_{k}(x,\omega).
\end{align*}
Then
\begin{align*}
\|\{ G_{j} \}_{j=0}^{\infty}\|_{L^{q(\cdot)}(\ell^{q(\cdot)}_{E})}
\leq C \|\{ g_{j} \}_{j=0}^{\infty}\|_{L^{q(\cdot)}(\ell^{q(\cdot)}_{E})},
\end{align*}
where
\begin{align*}
\|\{ g_{j} \}_{j=0}^{\infty}\|_{L^{q(\cdot)}(\ell^{q(\cdot)}_{E})} =
\left\| \left( \sum_{j=0}^{\infty} \mathbb{E}\left( g_{j}(\cdot)^{q(\cdot)} \right) \right)^{1/q(\cdot)} \right\|_{L^{q(\cdot)}(\mathbb{T}^{n})},
\end{align*}
and
\begin{align*}
\mathbb{E}\left( g_{j}(x)^{q(x)}\right) = \int_{\Omega} \left(g_{j}(x,\omega)\right)^{q(x)} \mathbb{P}(d\omega).
\end{align*}
\end{lemma}
\begin{proof}
It is obviously that we only need to give the following estimates
\begin{align*}
\left( \sum_{j} \mathbb{E}\left( g_{j}(x)^{q(x)} \right) \right)^{1/q(x)} &
\leq \left( \int_{\Omega} ( \sum_{j=0}^{\infty} 2^{-|j|\delta} )^{q(x)}
\sum_{j=0}^{\infty}|g_{j}(x,\omega)|^{q(x)} \mathbb{P}(d\omega) \right)^{1/q(x)}   \\
& \leq C \left( \sum_{j=0}^{\infty} \int_{\Omega} |g_{j}(x,\omega)|^{q(x)} \mathbb{P}(d\omega) \right)^{1/q(x)} \\
& \leq C \left( \sum_{j=0}^{\infty}\mathbb{E}\left( g_{j}(x) \right)^{q(x)} \right)^{1/q(x)}.
\end{align*}
\end{proof}
\begin{lemma}\label{basic_maximal_instead}
Let $q(\cdot) \in C^{log}(\mathbb{T}^{n})$ with $1 < q^{-} \leq q(\cdot) \leq q^{+} < \infty$. Then the inequality
\begin{align*}
\left\| \{ \eta_{v,R}*f_{j} \}_{j \in \mathbb{N}_{0}} \right\|_{L^{q(\cdot)}(\ell^{q(\cdot)}_{E})}
\leq C \left\| \{ f_{j} \}_{j \in \mathbb{N}_{0}} \right\|_{L^{q(\cdot)}(\ell^{q(\cdot)}_{E})}
\end{align*}
holds for every sequence $\{ f_{j}(x,\omega) \}_{j \in \mathbb{N}_{0}}$ of $L_{loc}^{1}$-functions for variable $x$
and $\mathbb{P}$-measurable functions for variable $\omega$.
\end{lemma}
\begin{proof}
The proof of this lemma similar to the proof Lemma 5.4 in \cite{variable_tribel}, so here we only give the different
part. Let $\mathcal{D}_{i}$ stands for all dyadic cubes with side length $2^{-i}$ and
$\eta_{\nu m}(x) = \frac{2^{n\nu}}{(1+2^{\nu}|x|)^m}.$
As in \cite{variable_tribel}, we need the following estimate
\begin{align*}
& \int_{\mathbb{T}^{n}} \sum_{\nu = 0}^{\infty} \int_{\Omega} |\eta_{\nu m} * f_{\nu} |^{q(x)} \mathbb{P}(d\omega) dx \\
\leq & \int_{\mathbb{T}^{n}} \sum_{\nu =0}^{\infty} \left( \int_{\Omega} \sum_{j = 0}^{\infty}
2^{-j(m-n)} \sum_{Q \in \mathcal{D}_{\nu-j}} \chi_{3Q}(x) M_{Q}f_{\nu} \mathbb{P}(d\omega) \right)^{q(x)} dx    \\
\leq & C \int_{\mathbb{T}^{n}} \sum_{\nu = 0}^{\infty} \int_{\Omega} \sum_{j = 0}^{\infty}
2^{-j(m-n)} \sum_{Q \in \mathcal{D}_{\nu-j}} \chi_{3Q}(x) \left( M_{Q}\left( |f_{\nu}|^{q(x)/q^{-}}\right) \right)^{q^{-}}
\mathbb{P}(d\omega) dx + II \\
= & I + II,
\end{align*}
where $II$ is exactly the same as in the proof Lemma 5.4 in \cite{variable_tribel}. Next, we only give estimate for term $I$.
\begin{align*}
I \leq & C \int_{\mathbb{T}^{n}} \sum_{\nu =0}^{\infty} \int_{\Omega} \left( M(|f_{\nu}|^{q(x)/q^{-}}) \right)^{q^{-}} \mathbb{P}(d\omega)
\sum_{j=0}^{\infty}2^{-j(m-n)} \sum_{Q \in \mathcal{D}_{\nu - j}}\chi_{3Q}(x) dx    \\
\leq & C \int_{\mathbb{T}^{n}} \sum_{\nu = 0}^{\infty} \mathbb{E} \left( |f_{\nu}(x)|^{q(x)} \right) dx < \infty.
\end{align*}
With these estimates, it is easy to recover the whole proof.
\end{proof}
With the two Lemmas in hand, following the proof of Theorem 4.4, Theorem 4.5, we can get the local
means characterizations of $L_{\mathbb{P}}^{q(\cdot)}(\Omega; B^{s(\cdot)}_{q(\cdot)})$ by using
our Lemma \ref{basic_measure} and Lemma \ref{basic_maximal_instead} instead of Lemma 4.2 and
Lemma 4.3 in \cite{2microlocal}. Replacing Lemma 5 and Lemma 9 in \cite{Kermpka} by our Lemma \ref{basic_measure}
and Lemma \ref{basic_maximal_instead}, we can mimic the proof of Corollary 2 and Corollary 3 in \cite{Kermpka} to
give the proof of Lemma \ref{wavelet_characterization_measure}.
Considering the proof is so long and have no new ingredient except Lemma \ref{basic_measure}
and Lemma \ref{basic_maximal_instead}, we omit it here.


\section{Appendix}

\begin{proposition}\label{definition_wavelet}\cite{Kermpka}
(i) There are a real scaling function $\varphi_{F} \in \mathcal{S}(\mathbb{R})$ and a real associated wavelet
$\varphi_{M} \in \mathcal{S}(\mathbb{R})$ such that their Fourier transforms have compact supports,
$\hat{\varphi}_{F}(0) = 1$ and
\begin{align*}
\mathrm{supp}\, \hat{\varphi}_{M} \subset \left[ -\frac{8}{3}\pi, -\frac{2}{3}\pi \right] \cup
\left[ \frac{2}{3}\pi, \frac{8}{3}\pi \right].
\end{align*}
(ii) For any $k \in \mathbb{N}$ there exist a real compactly supported scaling function $\varphi_{F} \in C^{k}(\mathbb{R})$ and
a real compactly supported associated wavelet $\varphi_{M} \in C^{k}(\mathbb{R})$ such that $\hat{\varphi}_{F}(0) = 1$ and
\begin{align*}
\int_{\mathbb{R}} x^{\ell} \varphi_{M}(x)dx = 0 \quad \text{for all }\ell \in \{ 0,1,\cdots, k-1 \}.
\end{align*}
In both cases we have that $\{\varphi_{vm}: v\in \mathbb{N}\cup{0}, m \in \mathbb{Z}\}$ is an orthonormal
basis in $L^{2}(\mathbb{R})$ where
\begin{align*}
\varphi_{vm}(t) := \left\{
\begin{array}{c}
  \varphi_{F}(t-m), \quad \text{if }v = 0, m\in \mathbb{Z} \\
  2^{\frac{v-1}{2}}\varphi_{M}(2^{v-1}t-m), \quad \text{if }v\in \mathbb{N}, m\in \mathbb{Z}
\end{array}
\right.
\end{align*}
and the functions $\varphi_{M}$, $\varphi_{F}$ are according to (i) or (ii).
\end{proposition}
The wavelets in the first part of the above proposition are called Meyer wavelets. They do not have
a compact support but they are fast decaying functions and $\varphi_{M}$ has infinitely many moment conditions.
The wavelets from the second part of the above proposition are called Daubechies wavelets.
Here the functions $\varphi_{M}, \varphi_{F}$ do have compact support, but they only have limited
smoothness.

\begin{definition}\label{atom_periodic}\cite{Kermpka}
Let $s(\cdot) \in L^{\infty} \cap C_{loc}^{log}(\mathbb{T}^{n})$, $0< q \leq \infty$ and $p(\cdot) \in \mathcal{P}(\mathbb{T}^{n})$
with $0 < p^{-} \leq p^{+} \leq \infty$.
Denote $\mathbb{M}_{j} = \{ m: m = 0,1,2,\cdots,2^j-1\}$.

(i) Then
\begin{align*}
\tilde{b}_{p(\cdot),q}^{s(\cdot)} := \left\{ \lambda = \{ \lambda_{Gm}^{j} \}_{j \in \mathbb{N}_{0}, G\in G^{j}, m \in \mathbb{M}_{j}}
: \, \| \lambda \|_{\tilde{b}_{p(\cdot),q}^{s(\cdot)}} < \infty \right\}
\end{align*}
where
\begin{align*}
\| \lambda \|_{\tilde{b}_{p(\cdot),q}^{s(\cdot)}} = \left( \sum_{j = 0}^{\infty} \sum_{G \in G^{j}}
\left\| \sum_{m \in \mathbb{M}_{j}} 2^{js(2^{-j}m)}|\lambda_{Gm}|^{j}\chi_{jm}(\cdot) \right\|_{L^{p(\cdot)}(\mathbb{T}^{n})}^{q} \right)^{1/q}.
\end{align*}

(ii) For $p^{+} < \infty$, we define
\begin{align*}
\tilde{f}_{p(\cdot),q(\cdot)}^{s(\cdot)} := \left\{ \lambda = \{ \lambda_{Gm}^{j} \}_{j \in \mathbb{N}_{0}, G \in G^{j}, m \in \mathbb{M}_{j}}
: \, \|\lambda\|_{\tilde{f}_{p(\cdot),q(\cdot)}^{s(\cdot)}} < \infty \right\},
\end{align*}
where
\begin{align*}
\|\lambda\|_{\tilde{f}_{p(\cdot),q(\cdot)}^{s(\cdot)}} = \left\| \left(
\sum_{j = 0}^{\infty} \sum_{G \in G^{j}} \sum_{m \in \mathbb{M}_{j}} 2^{jqs(2^{-j}m)} |\lambda_{Gm}^{j}|^{q(\cdot)}
\chi_{jm}(\cdot) \right)^{1/q(\cdot)} \right\|_{L^{p(\cdot)}(\mathbb{T}^{n})}
\end{align*}
with $q(\cdot) \in \mathcal{P}(\mathbb{T}^{n})$.
\end{definition}

\begin{theorem}\label{low_of_large_number}
Let $X_{1}, X_{2}, \cdots$ be pairwise independent identically distributed random variables with $\mathbb{E}|X_{i}| < \infty$.
Let $\mathbb{E}X_{i} = \mu$ and $S_{n} = X_{1} + X_{2} + \cdots + X_{n}$. Then
$S_{n}/n \rightarrow \mu$ a.s. as $n\rightarrow \infty$.
\end{theorem}

The setting is to consider a random function $u$ given by the random series
\begin{align}\label{appendex_series}
u = \sum_{k\geq 0}\xi_{k} \psi_{k}
\end{align}
where $\{\xi_{k}\}_{k}$ is an i.i.d. sequence and the $\psi_{k}$ are real- or complex-valued H\"{o}lder functions on founded
open $D \subset \mathbb{R}^{n}$ satisfying, for some $\alpha \in (0,1]$,
\begin{align}\label{appendex_holder}
|\psi_{k}(x) - \psi_{k}(y)| \leq h(\alpha, \psi_{k}) |x-y|^{\alpha} \quad x,y \in D;
\end{align}
of course if $\alpha = 1$ the functions are Lipschitz.
\begin{theorem}\label{kolmogorov_test}\cite{Dashti_Stuart}
Let $\{\xi_{k}\}_{k\geq 0}$ be countably many centered i.i.d. random variables with bounded moments of all orders. Moreover
let $\{\psi_{k}\}_{k\geq 0}$ satisfy (\ref{appendex_holder}). Suppose there is some $\delta \in (0,2)$ such that
\begin{align}\label{appendix_s1}
S_{1}:= \sum_{k\geq 0}\|\psi_{k}\|_{L^{\infty}}^{2} < \infty,
\end{align}
and
\begin{align}\label{appendex_s2}
S_{2} := \sum_{k \geq 0}\|\psi_{k}\|_{L^{\infty}}^{2-\delta} h(\alpha,\psi_{k})^{\delta} < \infty.
\end{align}
Then $u$ defined by (\ref{appendex_series}) is almost surely finite for every $x\in D$, and $u$ is H\"{o}lder
continuous for every H\"{o}lder exponent smaller than $\alpha \delta /2$.
\end{theorem}

\section{Acknowledgements}
J. Jia would like to thank financial support by Beijing Center for Mathematics and Information Interdisciplinary Science (BCMIIS)
J. Jia was supported by the National Natural Science Foundation of China under the grants no. 11131006, 41390450, and 91330204
and partially by the National Basic Research Program of China under the grant no. 2013CB329404.

\end{document}